\documentclass[12pt]{amsart}
\newtheorem{theorem}{Theorem}[section]

\newtheorem{proposition}[theorem]{Proposition}
\theoremstyle{definition}
\newtheorem{definition}[theorem]{Definition}

\usepackage{enumerate}
\theoremstyle{remark}

\usepackage[margin=0.75in]{geometry}
\usepackage{calligra}
\numberwithin{equation}{section}
\DeclareMathOperator*{\esssup}{ess\,sup}
\def\H{{\mathbb{H}^n}}
\def\L{{L^2(\mathbb{H}^n)}}
\def\N{{\mathbb{N}}}
\def\h{{\mathcal{H}}}
\def\R{{\mathbb{R}}}
\def\l{{L^2(\mathbb{R}^n)}}
\def\b{{\mathcal{B}_2}}
\def\Z{{\mathbb{Z}^n}}
\def\z{{\mathbb{Z}}}

\def\c{{\mathbb{C}}}

\newcommand\numberthis{\addtocounter{equation}{1}\tag{\theequation}}

\begin{document}
\title{Gabor System based on the unitary dual \\
of the Heisenberg group}

\author{S.R. Das}
\address{Department of Mathematics, Indian Institute of Technology, Madras}
\email{santiranjandas100@gmail.com}

\author{R. Radha}
\address{Department of Mathematics, Indian Institute of Technology, Madras}
\email{radharam@iitm.ac.in}





\begin{abstract}
In this paper Gabor system of certain type based on the unitary dual of the Heisenberg group $\H$ is introduced and a sufficient condition is obtained for the Gabor system to be a Bessel sequence for $L^2(\R^*,\b;d\kappa)$ using the {\it Schr\"{o}dinger} representation of $\H$, where $\b$ denotes the class of {\it Hilbert-Schmidt} operators on $L^2(\R^n)$ and $d\kappa$ denotes the Haar measure on $\R^*$. Further a necessary and sufficient condition is provided for the Gabor system to be an orthonormal system, a Parseval frame sequence, a frame sequence and a Riesz sequence. 
\end{abstract}
\maketitle
\section{Introduction}
The goal of Gabor analysis is to express a function in terms of time-frequency shifts of a single function, arising from two classes of operators on $L^2(\R)$, namely translation and modulation $T_a,M_b:L^2(\R)\rightarrow L^2(\R)$, defined by $(T_af)(x)=f(x-a)$ and $(M_bf)(x)=e^{2\pi ibx}f(x)$ respectively. If we restrict the translation and modulation parameters to a lattice $\{(mb,na):m,n\in\z\}\subset\R^2$, then the system $\{M_{mb}T_{na}g:m,n\in\z\}$ is called a Gabor system in $L^2(\R)$ generated by $g\in L^2(\R)$. Gabor systems are widely applied in signal processing, image analysis, digital communication, quantum physics and so on. Also, in theoretical point of view, it has been a basic question of interest under what condition on the window function $g\in L^2(\R)$ the Gabor system will be a frame, a Parseval frame, an orthonormal basis and a Riesz basis. There are several interesting and deep results available for Gabor system in $L^2(\R)$. For a detailed study of frames and Gabor system, we refer to the book \cite{olebook}.\\

In the recent years frames and Riesz bases for system of translates have been studied extensively in various group settings such as locally compact abelian group, compact non-abelian group, Heisenberg group and in general certain {\it Lie} groups. (See \cite{kamyabi}, \cite{cabrelli}, \cite{radhas}, \cite{CMO}, \cite{saswata}, \cite{iverson}, \cite{saswatacollec}, \cite{arati} and \cite{saswatahouston}.) The wavelet system has also been studied for a locally compact abelian group and Heisenberg group. (See \cite{dahlke}, \cite{holsch}, \cite{farkov}, \cite{gitta}, \cite{mayelimra}, \cite{araticol} and \cite{aratijmpa}.) However upto our knowledge Gabor system has not been studied on a non-compact non-abelian locally compact group. In this paper we attempt to study Gabor system on the Heisenberg group.\\

For a locally compact abelian group $G$, a Gabor system is defined as the collection $\{M_\nu T_\lambda g:\nu\in\Gamma, \lambda\in\Lambda\}$ where $\Gamma$ and $\Lambda$ are discrete co-compact subgroups of $\widehat{G}$ and $G$ respectively. (See \cite{jakobsen}.) Notice that for $\nu\in\widehat{G},~\lambda\in G$, both $M_\nu$ and $T_\lambda$ are unitary operators on $L^2(G)$. We shall try to look into a similar system on the Heisenberg group $\H$ using the appropriate lattices $L$ and $\Lambda$ in $\H$ and $\R^*$ (considered as the unitary dual of $\H$) respectively. In the case of non-abelian groups, it is natural to go for unitary irreducible representations in place of modulation. Thus for $g\in\L$, we have to look for $\pi_\lambda,~\lambda\in\Lambda\subset \R^*$ and the left translates $L_{(x^\prime,y^\prime,t^\prime)}$ with $(x^\prime,y^\prime,t^\prime)\in L\subset\H$. The operator $L_{(x^\prime,y^\prime,t^\prime)}$ is a unitary operator on $\L$. Further the representations $\pi_\lambda$, $\pi_\lambda (x,y,t),~(x,y,t)\in\H$, can act only on elements from $\l$. Thus $\pi_\lambda(x,y,t)L_{(x^\prime,y^\prime,t^\prime)}g$ is not meaningful.\\

There is another approach of defining Gabor system on the locally compact abelian group $G$. Christensen and Goh in \cite{ole} define such a system and call it Gabor-type system. It is of the form $\{M_\lambda T_k \Phi:\lambda\in\Lambda,k\in\Gamma\}$, for lattices $\Lambda$, $\Gamma$ in $G$ and $\widehat{G}$ respectively for a function $\Phi\in L^2(\widehat{G})$. Based on this idea we define the Gabor system on the Heisenberg group using its unitary dual $\R^*$. Recall that for a locally compact group $G$, a lattice $\Lambda$ in $G$ is defined to be a discrete subgroup of $G$ which is co-compact. The standard lattice in $\H$ is taken to be $\Lambda=\{(2k,l,m):k,l\in\Z,m\in\z\}$. A standard lattice in $\R^*$ is taken to be $L=\{e^{bp}:p\in\z\}$ for $b>0$. Based on this we define the Gabor system on the Heisenberg group to be the system $\{T_{e^{bp}}M_{a(2k,l,m)}\mathcal{G}:k,l\in\Z,m,p\in\z\}$ generated by $\mathcal{G}$ on the unitary dual $\R^*$, $a,b>0$. More precisely we take $\mathcal{G}\in L^2(\R^*,\b;d\kappa)$, where $\b$ denotes the {\it Hilbert} space of {\it Hilbert-Schmidt} operators on $L^2(\R^n)$ and $d\kappa$ denotes the {\it Haar} measure on $\R^*$.\\

We organize our paper as follows. In section \ref{x}, we give necessary notation and background. In section \ref{y}, we discuss the properties of the Fourier transform of $\b$ valued functions on the locally compact abelian group $\R^*$.  In section \ref{b}, our main aim is to obtain a sufficient condition for the Gabor system to be a Bessel sequence for $L^2(\R^*,\b;d\kappa)$ using the {\it Schr\"{o}dinger} representation of $\H$. In section \ref{z}, we provide a necessary and sufficient condition for the Gabor system to be an orthonormal system, a Parseval frame sequence, a frame sequence and a Riesz sequence.

\section{Preliminaries}\label{x}
Let $\h\neq\{0\}$ be a separable Hilbert space.
\begin{definition}
A sequence $\{f_k:~k\in\N\}$ of elements in $\h$ is said to be a frame for $\h$ if there exist constants $A,B>0$ such that
\begin{align}\label{32}
A\|f\|^2\leq\sum_{k\in\N}\mid\langle f,f_k\rangle\mid^2\hspace{1 mm}\leq B\|f\|^2,\hspace{.75 cm}\forall\ f\in\h.
\end{align}
\end{definition}
The numbers $A$ and $B$ are called lower and upper frame bounds respectively. If only the right hand side of \eqref{32} is satisfied then $\{f_k:~k\in\N\}$ is called a Bessel sequence. In particular if $A=B$ holds in \eqref{32} then $\{f_k:~k\in\N\}$ is called a tight frame. If \eqref{32} holds with $A=B=1$ then $\{f_k:~k\in\N\}$ is called a Parseval frame. If $\{f_k:~k\in\N\}$ is a frame for $\overline{span}\{f_k:~k\in\N\}$, then it is called a frame sequence.\\

Let $\{f_k:~k\in\N\}$ be a frame for $\h$. The operator $\mathbf{T}:\ell^2(\N)\rightarrow\h$ defined by 
\[\mathbf{T}(\{c_k\}_{k=1}^\infty)=\sum_{k=1}^\infty c_kf_k\hspace{1 mm},\hspace{.75 cm}\forall\hspace{1 mm}\{c_k\}_{k=1}^\infty\in \ell^2(\N),\]
is called the synthesis operator. The adjoint of $\mathbf{T}$ denoted by $\mathbf{T}^*:\h\rightarrow \ell^2(\N)$ can be written as 
\[\mathbf{T}^*f=\{\langle f,f_k\rangle\}_{k=1}^\infty,\hspace{1 mm}\forall\hspace{1 mm} f\in\h,\]  
called the analysis operator. The frame operator is defined to be $\mathbf{S}=\mathbf{T}\mathbf{T}^*$ which can be explicitly written as 
\[\mathbf{S}f=\sum_{k=1}^\infty\langle f,f_k\rangle f_k,\hspace{1 mm}\forall\hspace{1 mm} f\in\h.\]
It turns out that $\mathbf{S}$ is a bounded, self-adjoint, positive and invertible operator on $\h$.
\begin{definition}
Suppose $\{f_k:~k\in\N\}$ is a frame for $\h$. The frame $\{\mathbf{S}^{-1}f_k:k\in\N\}$ is called the canonical dual frame of $\{f_k:~k\in\N\}$, where $\mathbf{S}$ is the corresponding frame operator. 
\end{definition}
\begin{definition}
A sequence of the form $\{Ue_k:~k\in\N\}$, where $\{e_k:~k\in\N\}$ is an orthonormal basis of $\h$ and $U$ is a bounded invertible operator on $\h$, is called a Riesz basis. If $\{f_k:~k\in\N\}$ is a Riesz basis for $\overline{span}\{f_k:~k\in\N\}$, then it is called a Riesz sequence. 
\end{definition}
For a study on frames and Riesz bases on $\mathcal{H}$ we refer to \cite{olebook}.\\

Let $\H =\R^n\times\R^n\times\R$ with the group operation defined by 
\[(x,y,t)(u,v,s)=(x+u,y+v,t+s+\frac{1}{2}(u\cdot y-v\cdot x)).\]
Then $\H$ is called the Heisenberg group. The group $\H$ is an example of a non-abelian locally compact non compact $Lie ~group$. The $Haar$ measure on $\H$ is the usual product measure of $Lebesgue$ measures on $\R^{2n}$ and $\R$. The classical Stone-von Neumann theorem states that every infinite dimensional irreducible unitary representation of $\H$ is unitarily equivalent to the {\it Schr\"{o}dinger} representation $\pi_\lambda$, $\lambda\in\R^*$, given by
\begin{equation*}
\pi_\lambda (x,y,t)\phi (\xi)=e^{2\pi i\lambda t}e^{2\pi i\lambda (x\cdot\xi +\frac{1}{2} x\cdot y)}\phi(\xi +y),~~\phi\in\l.
\end{equation*} 
Therefore the unitary dual of $\H$ is $\R^*$. Recall that for a locally compact group $G$, a lattice $\Gamma$ in $G$ is defined to be a discrete subgroup of $G$ which is co-compact. The standard lattice in $\H$ is taken to be $\Gamma=\{(2k,l,m):k,l\in\Z,~m\in\z\}$.\\

The Weyl transform $W:L^2(\R^{2n})\rightarrow\b(L^2(\R^n))$ is defined by
\begin{align*}
W(g)\phi(\xi)=\int_{\R^{2n}}g(x,y)\pi(x,y)\phi(\xi)\ dxdy\ ,
\end{align*}
where $\b(L^2(\R^n))$ is the {\it Hilbert} space of all {\it Hilbert-Schmidt} operators on $L^2(\R^n)$ with the inner product $(T,S)=tr(TS^*)$. Here by $\pi(x,y)$ we mean $\pi_1(x,y,0)$. One can show that the Weyl transform is an isometric isomorphism of $L^2(\R^{2n})$ onto $\b(L^2(\R^n))$. The inversion formula for the Weyl transform is as follows:
\begin{align*}
g(u,v)=tr(\pi(u,v)^*W(g)).
\end{align*}
As in the classical case, the inversion formula is initially valid for the functions belonging to the \emph{Schwartz} space $\mathcal{S}(\R^{2n})$ and then extended to the whole of $L^2(\R^{2n})$. For the study of analysis on the Heisenberg group and the Weyl transform we refer to \cite{follandphase} and \cite{thangavelu}.
\section{Fourier Transform on $L^2(\R^*,\b;d\kappa)$}\label{y}
The Fourier analysis on a locally compact abelian group is well known. (See \cite{rudin} and \cite{follandabs}.) However in this section we first define the characters, Fourier transform and recollect its properties explicitly on the locally compact abelian group $\R^*$ before defining the Fourier transform on $L^2(\R^*,\b;d\kappa)$, where $d\kappa(\lambda)=\frac{d\lambda}{|\lambda|}$ is the Haar measure on the group $\R^*$. These are required for studying Gabor system on the Heisenberg group.  
Let $\sigma:\R^*\rightarrow\R_+^*\times\{1,-1\}$ be defined by $\sigma(\lambda)=\big(|\lambda|,\frac{\lambda}{|\lambda|}\big)$. Then $\sigma$ is an isomorphism.
It is easy to see that $\sigma$ is also a continuous map with continuous inverse. Therefore $\sigma$ is a group isomorphism and homeomorphism between $\R^*$ and $\R_+^*\times\{1,-1\}$. The collection of all distinct characters of $\R_+^*$ is $\{\chi_\nu:\nu\in\R_+^*\}$, where $\chi_\nu(\lambda)=e^{2\pi i\log\nu\log\lambda},\ \lambda\in\R_+^*$, and for the group $\{1,-1\}$ is $\{\varsigma_1,\varsigma_{-1}\}$, where $\varsigma_1(\pm 1)=1$ and $\varsigma_{-1}(\pm 1)=\pm 1$. Therefore the collection of all characters of $\R_+^*\times\{1,-1\}$ is $\{\chi_{\nu,j}:\nu\in\R_+^*,j\in\{\pm 1\}\}$, where $\chi_{\nu,j}(\lambda,\pm 1)=\chi_\nu(\lambda)\varsigma_j(\pm 1)$.\\

Let $f,g\in L^1(\R^*)$. Then for $(\nu,j)\in\R_+^*\times\{1,-1\}$, the group Fourier transform of $f$ is
\[\widehat{f}(\nu,j)=\int_{\R^*}f(\lambda)\ \overline{(\nu,j)(\lambda)}\ \frac{d\lambda}{|\lambda|}\ ,\]
where $(\nu,j)(\lambda)=\chi_{\nu,j}(|\lambda|,\frac{\lambda}{|\lambda|})$. We define the convolution of $f,g$ by
\begin{align*}
(f*g)(\lambda)=\int_{\R^*}f\bigg(\frac{\lambda}{\omega}\bigg)g(\omega)\ \frac{d\omega}{|\omega|}.
\end{align*}
For $(\nu,j)\in\R_+^*\times\{1,-1\}$, we have
\begin{align*}
\widehat{f*g}(\nu,j)&=\int_{\R^*}(f*g)(\lambda)\ \overline{(\nu,j)(\lambda)}\ \frac{d\lambda}{|\lambda|}\\
&=\int_{\R^*}\int_{\R^*}f\bigg(\frac{\lambda}{\omega}\bigg)g(\omega)\ \overline{(\nu,j)(\lambda)}\ \frac{d\omega}{|\omega|}\frac{d\lambda}{|\lambda|}\\
&=\int_{\R^*}\Bigg(\int_{\R^*}f\bigg(\frac{\lambda}{\omega}\bigg)\ \overline{(\nu,j)(\lambda)}\ \frac{d\lambda}{|\lambda|}\Bigg)g(\omega)\ \frac{d\omega}{|\omega|}\\
&=\int_{\R^*}\Bigg(\int_{\R^*}f(\lambda)\ \overline{(\nu,j)(\lambda\omega)}\ \frac{d\lambda}{|\lambda|}\Bigg)g(\omega)\ \frac{d\omega}{|\omega|}\ ,
\end{align*}
by applying Fubini's theorem and change of variables. Since characters are homomorphisms, we have
\begin{align*}
\widehat{f*g}(\nu,j)&=\int_{\R^*}\Bigg(\int_{\R^*}f(\lambda)\ \overline{(\nu,j)(\lambda)}\ \frac{d\lambda}{|\lambda|}\Bigg)g(\omega)\ \overline{(\nu,j)(\omega)}\ \frac{d\omega}{|\omega|}\\
&=\widehat{f}(\nu,j)\ \widehat{g}(\nu,j).\numberthis\label{11}
\end{align*}
For $\omega\in\R^*$ we define the translation operator $T_\omega:L^1(\R^*)\longrightarrow L^1(\R^*)$ by $T_\omega h(\lambda)=h\big(\frac{\lambda}{\omega}\big),\ \forall\ h\in L^1(\R^*),\lambda\in\R^*$. Now
\begin{align*}
\widehat{T_\omega f}(\nu,j)&=\int_{\R^*}T_\omega f(\lambda)\ \overline{(\nu,j)(\lambda)}\ \frac{d\lambda}{|\lambda|}\\
&=\int_{\R^*}f\bigg(\frac{\lambda}{\omega}\bigg)\ \overline{(\nu,j)(\lambda)}\ \frac{d\lambda}{|\lambda|}\\
&=\int_{\R^*}f(\lambda)\ \overline{(\nu,j)(\lambda\omega)}\ \frac{d\lambda}{|\lambda|}\\
&=\overline{(\nu,j)(\omega)}\int_{\R^*}f(\lambda)\ \overline{(\nu,j)(\lambda)}\ \frac{d\lambda}{|\lambda|}\\
&=\overline{(\nu,j)(\omega)}\ \widehat{f}(\nu,j)\ ,
\end{align*}
by applying change of variables and homomorphism properties of the characters. The inversion formula is as follows:
\begin{align*}
f(\lambda)&=\int_{\R_+^*\times\{1,-1\}}\widehat{f}(\nu,j)\ (\nu,j)(\lambda)\ \frac{d\nu}{\nu}dj\ ,
\end{align*}
where $dj$ is the counting measure on $\{1,-1\}$. Notice that, as in the classical case, the inversion formula is valid in a pointwise sense initially for an appropriate dense subspace and then later extended to the whole of $L^2(\R^*)$. In other words the Fourier inversion formula can be written more explicitly as
\begin{align*}
f(\lambda)&=\sum_{j\in\{1,-1\}}\int_{\R_+^*}\widehat{f}(\nu,j)\ (\nu,j)(\lambda)\ \frac{d\nu}{\nu}.
\end{align*}
Let $f\in L^1\cap L^2(\R^*)$. Define $\widetilde{f}:\R^*\longrightarrow\c$ by $\widetilde{f}(\lambda)=\overline{f(\frac{1}{\lambda})},\ \lambda\in\R^*$. Let $g=\widetilde{f}*f$. Then $g\in L^1(\R^*)$ and from \eqref{11} we have $\widehat{g}(\nu,j)=\widehat{\widetilde{f}}(\nu,j)\widehat{f}(\nu,j),\ \forall\ (\nu,j)\in\R_+^*\times\{1,-1\}$. Now
\begin{align*}
\widehat{\widetilde{f}}(\nu,j)&=\int_{\R^*}\widetilde{f}(\lambda)\ \overline{(\nu,j)(\lambda)}\ \frac{d\lambda}{|\lambda|}\\
&=\int_{\R^*}\overline{f\bigg(\frac{1}{\lambda}\bigg)}\ \overline{(\nu,j)(\lambda)}\ \frac{d\lambda}{|\lambda|}\\
&=\int_{\R^*}\overline{f(\lambda)}\ \overline{(\nu,j)\bigg(\frac{1}{\lambda}\bigg)}\ \frac{d\lambda}{|\lambda|}\ ,\numberthis\label{12}
\end{align*}
by applying change of variables. But
\begin{align*}
(\nu,j)\bigg(\frac{1}{\lambda}\bigg)&=\chi_{\nu,j}\bigg(\frac{1}{|\lambda|},\frac{|\lambda|}{\lambda}\bigg)=\chi_\nu\bigg(\frac{1}{|\lambda|}\bigg)\varsigma_j\bigg(\frac{|\lambda|}{\lambda}\bigg)=e^{2\pi i\log\nu\log(1/|\lambda|)}\varsigma_j\bigg(\frac{\lambda}{|\lambda|}\bigg)=e^{-2\pi i\log\nu\log|\lambda|}\varsigma_j\bigg(\frac{\lambda}{|\lambda|}\bigg)\\
&=\overline{\chi_\nu(|\lambda|)\varsigma_j\bigg(\frac{\lambda}{|\lambda|}\bigg)}=\overline{\chi_{\nu,j}\bigg(|\lambda|,\frac{\lambda}{|\lambda|}\bigg)}=\overline{(\nu,j)(\lambda)}.\numberthis\label{18}
\end{align*}
Hence \eqref{12} gives
\begin{align*}
\widehat{\widetilde{f}}(\nu,j)&=\int_{\R^*}\overline{f(\lambda)}\ (\nu,j)(\lambda)\ \frac{d\lambda}{|\lambda|}=\overline{\widehat{f}(\nu,j)}.
\end{align*}
Thus $\widehat{g}(\nu,j)=|\widehat{f}(\nu,j)|^2$. Now
\begin{align*}
\|f\|_{L^2(\R^*)}^2&=\int_{\R^*}|f(\lambda)|^2\ \frac{d\lambda}{|\lambda|}=\int_{\R^*}\widetilde{f}\bigg(\frac{1}{\lambda}\bigg)f(\lambda)\ \frac{d\lambda}{|\lambda|}=(\widetilde{f}*f)(1)=g(1).
\end{align*}
Hence by inversion formula and using the fact that $(\nu,j)(1)=1$, we get
\begin{align*}
\|f\|_{L^2(\R^*)}^2&=\sum_{j\in\{1,-1\}}\int_{\R_+^*}\widehat{g}(\nu,j)\ (\nu,j)(1)\ \frac{d\nu}{\nu}=\sum_{j\in\{1,-1\}}\int_{\R_+^*}|\widehat{f}(\nu,j)|^2\ \frac{d\nu}{\nu}\\
&=\int_{\R_+^*\times\{1,-1\}}|\widehat{f}(\nu,j)|^2\ \frac{d\nu}{\nu}dj=\|\widehat{f}\|_{L^2(\R_+^*\times\{1,-1\})}^2\\
&=\|\widehat{f}\|_{L^2(\R^*)}^2\ ,
\end{align*}
proving Plancherel formula.\\

Now we shall define the Fourier transform on $L^2(\R^*,\b;d\kappa)$, where $\b$ is the set of all Hilbert-Schmidt operators on $L^2(\R^n)$. Let $f,g\in L^1(\R^*,\b;d\kappa)$. Then for $(\nu,j)\in\R_+^*\times\{1,-1\}$,
\begin{align*}
\widehat{f}(\nu,j)=\int_{\R^*}\overline{(\nu,j)(\lambda)}\ f(\lambda)\ \frac{d\lambda}{|\lambda|}\ ,
\end{align*}
where the integral is a Bochner integral. The convolution of $f$ and $g$ is defined by
\begin{align*}
(f*g)(\lambda)=\int_{\R^*}f\bigg(\frac{\lambda}{\omega}\bigg)g(\omega)\ \frac{d\omega}{|\omega|}\ ,
\end{align*}
where this integral is also a Bochner integral. Here the product $f\big(\frac{\lambda}{\omega}\big)g(\omega)$ is the composition of two Hilbert-Schmidt operators. Since $\b$ is not commutative, $f*g\neq g*f$ in general. For $\omega\in\R^*$ the translation operator $T_{\omega}:L^1(\R^*,\b;d\kappa)\longrightarrow L^1(\R^*,\b;d\kappa)$ is defined as $T_{\omega}h(\lambda)=h\big(\frac{\lambda}{\omega}\big)$. As in the case of Fourier transform on $\R^*$, we can show that $\widehat{f*g}(\nu,j)=\widehat{f}(\nu,j)\widehat{g}(\nu,j)$ and $\widehat{T_\omega f}(\nu,j)=\overline{(\nu,j)(\omega)}\ \widehat{f}(\nu,j)$, $\forall\ (\nu,j)\in\R_+^*\times\{1,-1\}$. The inversion formula is as follows:
\begin{align*}
f(\lambda)=\sum_{j\in\{1,-1\}}\int_{\R_+^*}(\nu,j)(\lambda)\ \widehat{f}(\nu,j)\ \frac{d\nu}{\nu}.
\end{align*} 

Let $f\in L^1\cap L^2(\R^*,\b;d\kappa)$. Define $\widetilde{f}:\R^*\longrightarrow\c$ by $\widetilde{f}(\lambda)=f\big(\frac{1}{\lambda}\big)^*$, where $f\big(\frac{1}{\lambda}\big)^*$ denotes the adjoint of the Hilbert-Schmidt operator $f\big(\frac{1}{\lambda}\big)$. Define $g=\widetilde{f}*f$. Now
\begin{align*}
\widehat{\widetilde{f}}(\nu,j)&=\int_{\R^*}\widetilde{f}(\lambda)\ \overline{(\nu,j)(\lambda)}\ \frac{d\lambda}{|\lambda|}=\int_{\R^*}f\bigg(\frac{1}{\lambda}\bigg)^*\ \overline{(\nu,j)(\lambda)}\ \frac{d\lambda}{|\lambda|}=\int_{\R^*}f(\lambda)^*\ \overline{(\nu,j)\bigg(\frac{1}{\lambda}\bigg)}\ \frac{d\lambda}{|\lambda|}\\
&=\int_{\R^*}f(\lambda)^*\ (\nu,j)(\lambda)\ \frac{d\lambda}{|\lambda|}=\Bigg(\int_{\R^*}f(\lambda)\ \overline{(\nu,j)(\lambda)}\ \frac{d\lambda}{|\lambda|}\Bigg)^*\\
&=\big(\widehat{f}(\nu,j)\big)^*\ ,
\end{align*}
by applying change of variables and using \eqref{18}. For an orthonormal basis $\{\phi_s:s\in\z\}$ of $L^2(\R^n)$, we have
\begin{align*}
\|f\|_{L^2(\R^*,\b;d\kappa)}^2&=\int_{\R^*}\|f(\lambda)\|_\b^2\ \frac{d\lambda}{|\lambda|}=\int_{\R^*}tr\big(f(\lambda)^*f(\lambda)\big)\ \frac{d\lambda}{|\lambda|}\\
&=\int_{\R^*}\sum_{s\in\z}\langle\phi_s,f(\lambda)^*f(\lambda)\phi_s\rangle\ \frac{d\lambda}{|\lambda|}\\
&=\sum_{s\in\z}\Bigg\langle\phi_s,\Bigg(\int_{\R^*}\widetilde{f}\bigg(\frac{1}{\lambda}\bigg)^*f(\lambda)\ \frac{d\lambda}{|\lambda|}\Bigg)\phi_s\Bigg\rangle\\
&=\sum_{s\in\z}\big\langle\phi_s,(\widetilde{f}*f)(1)\phi_s\big\rangle=tr\big(g(1)\big).
\end{align*}
Hence by inversion formula, we get
\begin{align*}
\|f\|_{L^2(\R^*,\b;d\kappa)}^2&=tr\Bigg(\sum_{j\in\{1,-1\}}\int_{\R_+^*}\widehat{g}(\nu,j)\ \frac{d\nu}{\nu}\Bigg)\\
&=\sum_{s\in\z}\Bigg\langle\phi_s,\Bigg(\sum_{j\in\{1,-1\}}\int_{\R_+^*}\widehat{g}(\nu,j)\ \frac{d\nu}{\nu}\Bigg)\phi_s\Bigg\rangle\\
&=\sum_{j\in\{1,-1\}}\int_{\R_+^*}\sum_{s\in\z}\big\langle\phi_s,\widehat{g}(\nu,j)\phi_s\big\rangle\ \frac{d\nu}{\nu}\\
&=\sum_{j\in\{1,-1\}}\int_{\R_+^*}tr\big(\widehat{g}(\nu,j)\big)\ \frac{d\nu}{\nu}\\
&=\sum_{j\in\{1,-1\}}\int_{\R_+^*}tr\bigg(\big(\widehat{f}(\nu,j)\big)^*\widehat{f}(\nu,j)\bigg)\ \frac{d\nu}{\nu}\\
&=\sum_{j\in\{1,-1\}}\int_{\R_+^*}\|\widehat{f}(\nu,j)\|_\b^2\ \frac{d\nu}{\nu}\\
&=\|\widehat{f}\|_{L^2(\R^*,\b;d\kappa)}^2\ ,
\end{align*}
proving Plancherel formula. By polarization identity we can show that
\begin{align}\label{25}
\big\langle f,g\big\rangle_{L^2(\R^*,\b;d\kappa)}=\big\langle \widehat{f},\widehat{g}\big\rangle_{L^2(\R^*,\b;d\kappa)}.
\end{align}      
\section{A Sufficient Condition for the Gabor System to be a Bessel Sequence}\label{b}
For $(x,y,t)\in\H$, we define $M_{(x,y,t)}:L^2(\R^*,\b;d\kappa)\longrightarrow L^2(\R^*,\b;d\kappa)$ by $(M_{(x,y,t)}F)(\lambda)=\pi_\lambda (x,y,t)F(\lambda)$, $F\in L^2(\R^*,\b;d\kappa),~\lambda\in\R^*$. Now
\begin{align*}
\|M_{(x,y,t)}F\|_{L^2(\R^*,\b;d\kappa)}^2&=\int_{\R^*}\|\pi_\lambda (x,y,t)F(\lambda)\|_\b^2~d\kappa\\
&=\int_{\R^*}\langle\pi_\lambda (x,y,t)F(\lambda),~\pi_\lambda (x,y,t)F(\lambda)\rangle_\b~d\kappa\\
&=\int_{\R^*}tr(F(\lambda)^*\pi_\lambda(x,y,t)^*\pi_\lambda(x,y,t)F(\lambda))~d\kappa\\
&=\int_{\R^*}tr(F(\lambda)^*F(\lambda))~d\kappa\\
&=\int_{\R^*}\|F(\lambda)\|_\b^2~d\kappa\\
&=\|F\|_{L^2(\R^*,\b;d\kappa)}^2.
\end{align*}    
In other words $M_{(x,y,t)}$ is an isometric isomorphism on $L^2(\R^*,\b;d\kappa)$. For $\lambda\in\R^*$, we define the translation operator $T_\lambda:L^2(\R^*,\b;d\kappa)\longrightarrow L^2(\R^*,\b;d\kappa)$ by $(T_\lambda F)(\lambda^\prime)=F(\frac{\lambda^\prime}{\lambda})$, $F\in L^2(\R^*,\b;d\kappa)$, $\lambda^\prime\in\R^*$. Using the fact that the Haar measure $d\kappa$ on $\R^*$ is invariant under translation, we can show that $T_\lambda$ is also an isometric isomorphism on $L^2(\R^*,\b;d\kappa)$.\\

Let $\mathcal{G}\in L^2(\R^*,\b;d\kappa)$ and $a,b>0$. We define the Gabor system generated by $\mathcal{G}$ as $\mathcal{U}(\mathcal{G})=\{T_{e^{bp}}M_{a(2k,l,m)}\mathcal{G}:k,l\in\Z,m,p\in\z\}$.
\begin{theorem}
Let $\mathcal{G}\in L^2(\R^*,\b;d\kappa)$ and $a,b>0$. Define $H_{k,l,m}(z,\lambda):=e^{2\pi ia\lambda(m-k\cdot y)}tr\big(\pi(2a\lambda k-x,al-y,0)\mathcal{G}(\lambda)\big)$, $\lambda\in\R^*,z=(x,y)\in\c^n,~k,l\in\Z,m\in\z$. Let $g_{k,l,m}(z,\nu,j)=e^{\pi ial\cdot x}\mathcal{F}_2H_{k,l,m}(z,\cdot)$ $(\nu,j)$, where $\mathcal{F}_2$ denotes the group Fourier transform on $\R^*$ in the second variable and $j\in\{1,-1\}$. Let
\begin{equation*}
\alpha(z,z^\prime)=\sum_{j,j^\prime\in\{1,-1\}}\esssup_{\nu\in\R_+^*}\sum_{\substack{k,l\in\Z,m\in\z \\ s\in\z\setminus \{0\}}}|g_{k,l,m}(z,\nu,j)|~|g_{k,l,m}(z^\prime,\nu e^{s/b},j^\prime)|~,
\end{equation*}
$z,z^\prime\in\c^n$. Suppose there exist $B_1,B_2$ such that
\begin{enumerate}[(i)]
\item $g_{k,l,m}$ satisfies 
\begin{equation*}
\esssup_{\nu\in\R_+^*}\sum_{k,l\in\Z,m\in\z}\|g_{k,l,m}(\cdot,\nu,j)\|_{L^2(\c^n)}^2\leq B_1,~~~j\in\{1,-1\}
\end{equation*}
and
\item $\alpha$ satisfies
\begin{equation*}
\esssup_{z\in\c^n}\|\alpha(z,\cdot)\|_{L^1(\c^n)}\leq B_2.
\end{equation*} 
\end{enumerate} 
Then the Gabor system $\mathcal{U}(\mathcal{G})$ is a Bessel sequence for $L^2(\R^*,\b;d\kappa)$.
\end{theorem}
\begin{proof}
Let $F\in L^2(\R^*,\b;d\kappa)$ and $\mathcal{G}_{k,l,m}=M_{a(2k,l,m)}\mathcal{G}$. Then by \eqref{25}
\begin{align*}
\langle F,T_{e^{bp}}M_{a(2k,l,m)}\mathcal{G}\rangle&=\langle \widehat{F},(T_{e^{bp}}M_{a(2k,l,m)}\mathcal{G})^{\widehat{}}~\rangle\\
&=\sum_{j\in\{1,-1\}}\int_{\R_+^*}\langle \widehat{F}(\nu,j),(T_{e^{bp}}M_{a(2k,l,m)}\mathcal{G})^{\widehat{}}~(\nu,j)\rangle~\frac{d\nu}{\nu}\\
&=\sum_{j\in\{1,-1\}}\int_{\R_+^*}\langle \widehat{F}(\nu,j),(M_{a(2k,l,m)}\mathcal{G})^{\widehat{}}~(\nu,j)\rangle(\nu,j)(e^{bp})~\frac{d\nu}{\nu}.
\end{align*}
From the definition of the characters, we have $(\nu,j)(e^{bp})=e^{2\pi ibp\log\nu}$, for $j\in\{1,-1\}$. Hence
\begin{align*}
\langle F,T_{e^{bp}}M_{a(2k,l,m)}\mathcal{G}\rangle&=\sum_{j\in\{1,-1\}}\int_{\R_+^*}\langle \widehat{F}(\nu,j),(M_{a(2k,l,m)}\mathcal{G})^{\widehat{}}~(\nu,j)\rangle e^{2\pi ibp\log\nu}~\frac{d\nu}{\nu}\\
&=\sum_{j\in\{1,-1\}}\sum_{s\in\z}\int_{e^{s/b}}^{e^{(s+1)/b}}\langle \widehat{F}(\nu,j),(M_{a(2k,l,m)}\mathcal{G})^{\widehat{}}~(\nu,j)\rangle e^{2\pi ibp\log\nu}~\frac{d\nu}{\nu}.
\end{align*}
By applying change of variables and taking summation inside the integral, we get
\begin{align*}
\langle F,T_{e^{bp}}M_{a(2k,l,m)}\mathcal{G}\rangle&=\sum_{j\in\{1,-1\}}\sum_{s\in\z}\int_1^{e^{1/b}}\langle \widehat{F}(\nu e^{s/b},j),(M_{a(2k,l,m)}\mathcal{G})^{\widehat{}}~(\nu e^{s/b},j)\rangle e^{2\pi ibp\log(\nu e^{s/b})}~\frac{d\nu}{\nu}\\
&=\int_1^{e^{1/b}}\sum_{j\in\{1,-1\}}\sum_{s\in\z}\langle \widehat{F}(\nu e^{s/b},j),(M_{a(2k,l,m)}\mathcal{G})^{\widehat{}}~(\nu e^{s/b},j)\rangle e^{2\pi ibp\log(\nu e^{s/b})}~\frac{d\nu}{\nu}\\
&=\int_1^{e^{1/b}}\sum_{j\in\{1,-1\}}\sum_{s\in\z}\langle \widehat{F}(\nu e^{s/b},j),(M_{a(2k,l,m)}\mathcal{G})^{\widehat{}}~(\nu e^{s/b},j)\rangle e^{2\pi ibp\log\nu}~\frac{d\nu}{\nu}\\
&=\int_1^{e^{1/b}}\Phi_{k,l,m}(\nu)e^{2\pi ibp\log\nu}~\frac{d\nu}{\nu}\ ,
\end{align*}
say. Now by applying change of variables, we get
\begin{align*}
\langle F,T_{e^{bp}}M_{a(2k,l,m)}\mathcal{G}\rangle&=\int_0^{1/b}\Psi_{k,l,m}(\nu)e^{2\pi ibp\nu}~d\nu~,
\end{align*}
where $\Psi_{k,l,m}(\nu)=\Phi_{k,l,m}(e^{\nu})$. Thus $\langle F,T_{e^{bp}}M_{a(2k,l,m)}\mathcal{G}\rangle=\frac{1}{b}\widehat{\Psi}_{k,l,m}(-p)$. Using the Plancherel formula of $L^2(0,1/b)$, we get
\begin{align*}
\sum_{p\in\z}|\langle F,T_{e^{bp}}M_{a(2k,l,m)}\mathcal{G}\rangle|^2&=\frac{1}{b^2}\sum_{p\in\z}|\widehat{\Psi}_{k,l,m}(-p)|^2\\
&=\frac{1}{b}\int_0^{1/b}|\Psi_{k,l,m}(\nu)|^2~d\nu\\
&=\frac{1}{b}\int_0^{1/b}|\Phi_{k,l,m}(e^{\nu})|^2~d\nu\\
&=\frac{1}{b}\int_1^{e^{1/b}}|\Phi_{k,l,m}(\nu)|^2~\frac{d\nu}{\nu}~,
\end{align*}
by applying change of variables. Now
\begin{align*}
&b\sum_{p\in\z}|\langle F,T_{e^{bp}}M_{a(2k,l,m)}\mathcal{G}\rangle|^2=\int_1^{e^{1/b}}\Phi_{k,l,m}(\nu)\overline{\Phi_{k,l,m}(\nu)}~\frac{d\nu}{\nu}\\
&=\int_1^{e^{1/b}}\hspace{-5 mm}\sum_{j,j^\prime\in\{1,-1\}}\sum_{s,s^\prime\in\z}\langle \widehat{F}(\nu e^{s/b},j),(M_{a(2k,l,m)}\mathcal{G})^{\widehat{}}~(\nu e^{s/b},j)\rangle
\overline{\langle \widehat{F}(\nu e^{s^\prime/b},j^\prime),(M_{a(2k,l,m)}\mathcal{G})^{\widehat{}}~(\nu e^{s^\prime/b},j^\prime)\rangle}~\frac{d\nu}{\nu}\\
&=\sum_{s\in\z}\int_1^{e^{1/b}}\hspace{-6 mm}\sum_{j,j^\prime\in\{1,-1\}}\hspace{-4 mm}\langle \widehat{F}(\nu e^{s/b},j),(M_{a(2k,l,m)}\mathcal{G})^{\widehat{}}~(\nu e^{s/b},j)\rangle\sum_{s^\prime\in\z}\overline{\langle \widehat{F}(\nu e^{s^\prime/b},j^\prime),(M_{a(2k,l,m)}\mathcal{G})^{\widehat{}}~(\nu e^{s^\prime/b},j^\prime)\rangle}~\frac{d\nu}{\nu}\\
&=\sum_{s\in\z}\hspace{-0.5 mm}\int_{e^{\frac{s}{b}}}^{e^{\frac{s+1}{b}}}\hspace{-6 mm}\sum_{j,j^\prime\in\{1,-1\}}\hspace{-4 mm}\langle \widehat{F}(\nu,j),(M_{a(2k,l,m)}\mathcal{G})^{\widehat{}}~(\nu,j)\rangle\sum_{s^\prime\in\z}\overline{\langle \widehat{F}(\nu e^{(s^\prime-s)/b},j^\prime),(M_{a(2k,l,m)}\mathcal{G})^{\widehat{}}~(\nu e^{(s^\prime-s)/b},j^\prime)\rangle}~\frac{d\nu}{\nu},
\end{align*}
using change of variables. Thus
\begin{align*}
&b\sum_{p\in\z}|\langle F,T_{e^{bp}}M_{a(2k,l,m)}\mathcal{G}\rangle|^2\\
&=\sum_{s\in\z}\int_{e^{s/b}}^{e^{(s+1)/b}}\hspace{-5 mm}\sum_{j,j^\prime\in\{1,-1\}}\hspace{-4 mm}\langle \widehat{F}(\nu,j),(M_{a(2k,l,m)}\mathcal{G})^{\widehat{}}~(\nu,j)\rangle\sum_{s^\prime\in\z}\overline{\langle \widehat{F}(\nu e^{s^\prime/b},j^\prime),(M_{a(2k,l,m)}\mathcal{G})^{\widehat{}}~(\nu e^{s^\prime/b},j^\prime)\rangle}~\frac{d\nu}{\nu}\\
&=\int_{\R_+^*}\sum_{j,j^\prime\in\{1,-1\}}\langle \widehat{F}(\nu,j),(M_{a(2k,l,m)}\mathcal{G})^{\widehat{}}~(\nu,j)\rangle\sum_{s^\prime\in\z}\overline{\langle \widehat{F}(\nu e^{s^\prime/b},j^\prime),(M_{a(2k,l,m)}\mathcal{G})^{\widehat{}}~(\nu e^{s^\prime/b},j^\prime)\rangle}~\frac{d\nu}{\nu}\\
&=\int_{\R_+^*}\sum_{j,j^\prime\in\{1,-1\}}\langle \widehat{F}(\nu,j),(M_{a(2k,l,m)}\mathcal{G})^{\widehat{}}~(\nu,j)\rangle\overline{\langle \widehat{F}(\nu,j^\prime),(M_{a(2k,l,m)}\mathcal{G})^{\widehat{}}~(\nu,j^\prime)\rangle}~\frac{d\nu}{\nu}~+\\
&\hspace{1 cm}\int_{\R_+^*}\sum_{j,j^\prime\in\{1,-1\}}\sum_{s^\prime\in\z\setminus\{0\}}\langle \widehat{F}(\nu,j),(M_{a(2k,l,m)}\mathcal{G})^{\widehat{}}~(\nu,j)\rangle\overline{\langle \widehat{F}(\nu e^{s^\prime/b},j^\prime),(M_{a(2k,l,m)}\mathcal{G})^{\widehat{}}~(\nu e^{s^\prime/b},j^\prime)\rangle}~\frac{d\nu}{\nu}\\
&=\sum_{j,j^\prime\in\{1,-1\}}\int_{\R_+^*}\langle \widehat{F}(\nu,j),(M_{a(2k,l,m)}\mathcal{G})^{\widehat{}}~(\nu,j)\rangle\overline{\langle \widehat{F}(\nu,j^\prime),(M_{a(2k,l,m)}\mathcal{G})^{\widehat{}}~(\nu,j^\prime)\rangle}~\frac{d\nu}{\nu}~+\\
&\hspace{1 cm}\sum_{j,j^\prime\in\{1,-1\}}\sum_{s^\prime\in\z\setminus\{0\}}\int_{\R_+^*}\langle \widehat{F}(\nu,j),(M_{a(2k,l,m)}\mathcal{G})^{\widehat{}}~(\nu,j)\rangle\overline{\langle \widehat{F}(\nu e^{s^\prime/b},j^\prime),(M_{a(2k,l,m)}\mathcal{G})^{\widehat{}}~(\nu e^{s^\prime/b},j^\prime)\rangle}~\frac{d\nu}{\nu}\\
&=\sum_{j,j^\prime\in\{1,-1\}}\theta_{k,l,m,j,j^\prime}^{(1)}~~+\sum_{j,j^\prime\in\{1,-1\}}\theta_{k,l,m,j,j^\prime}^{(2)}~,\numberthis \label{1}
\end{align*}
say. Now for $j\in\{1,-1\}$,
\begin{align*}
\sum_{k,l\in\Z,m\in\z}\theta_{k,l,m,j,j}^{(1)}&=\sum_{k,l\in\Z,m\in\z}\int_{\R_+^*}|\langle\widehat{F}(\nu,j),(M_{a(2k,l,m)}\mathcal{G})^{\widehat{}}~(\nu,j)\rangle|^2~\frac{d\nu}{\nu}\\
&=\sum_{k,l\in\Z,m\in\z}\int_{\R_+^*}|\langle\widehat{F}(\nu,j),\widehat{\mathcal{G}}_{k,l,m}~(\nu,j)\rangle|^2~\frac{d\nu}{\nu}.
\end{align*}
We can write $\widehat{F}(\nu,j)=W(f(\cdot,\nu,j)),\ \widehat{\mathcal{G}}_{k,l,m}~(\nu,j)=W(g_{k,l,m}(\cdot,\nu,j))\in\b(L^2(\R^n))$ for unique $f(\cdot,\nu,j),\ g_{k,l,m}(\cdot,\nu,j)\in L^2(\R^{2n})$, as the Weyl transform $W$ is an isometric isomorphism of $L^2(\R^{2n})$ onto $\b(L^2(\R^n))$. Thus
\begin{align*}
\sum_{k,l\in\Z,m\in\z}\theta_{k,l,m,j,j}^{(1)}&=\sum_{k,l\in\Z,m\in\z}\int_{\R_+^*}|\langle W(f(\cdot,\nu,j)),W(g_{k,l,m}(\cdot,\nu,j))\rangle|^2~\frac{d\nu}{\nu}.\numberthis\label{10}
\end{align*} 
By inversion formula for the Weyl transform, we can write
\begin{align*}
g_{k,l,m}(z,\nu,j)=tr(\pi (z)^*W(g_{k,l,m}(\cdot,\nu,j)))=tr(\pi (z)^*\widehat{\mathcal{G}}_{k,l,m}~(\nu,j)).
\end{align*}
Now for the orthonormal basis $\{\phi_s:s\in\z\}$ in $\l$, we have
\begin{align*}
g_{k,l,m}(z,\nu,j)&=\sum_{s\in\z}\langle\phi_s,\pi (z)^*\widehat{\mathcal{G}}_{k,l,m}~(\nu,j)\phi_s\rangle\\
&=\sum_{s\in\z}\Bigg\langle\phi_s,\pi (z)^*\Bigg(\int_{\R^*}\mathcal{G}_{k,l,m}(\lambda)\ \overline{(\nu,j)(\lambda)}\ \frac{d\lambda}{|\lambda|}\Bigg)\phi_s\Bigg\rangle\\
&=\sum_{s\in\z}\int_{\R^*}\langle\phi_s,\pi(z)^*\mathcal{G}_{k,l,m}(\lambda)\phi_s\rangle\ \overline{(\nu,j)(\lambda)}\ \frac{d\lambda}{|\lambda|}\\
&=\int_{\R^*}\sum_{s\in\z}\langle\phi_s,\pi(z)^*\mathcal{G}_{k,l,m}(\lambda)\phi_s\rangle\ \overline{(\nu,j)(\lambda)}\ \frac{d\lambda}{|\lambda|}\\
&=\int_{\R^*}tr(\pi(z)^*\mathcal{G}_{k,l,m}(\lambda))\ \overline{(\nu,j)(\lambda)}\ \frac{d\lambda}{|\lambda|}\\
&=\int_{\R^*}tr(\pi(z)^*M_{a(2k,l,m)}\mathcal{G}(\lambda))\ \overline{(\nu,j)(\lambda)}\ \frac{d\lambda}{|\lambda|}\\
&=\int_{\R^*}tr(\pi(z)^*\pi_\lambda(a(2k,l,m))\mathcal{G}(\lambda))\ \overline{(\nu,j)(\lambda)}\ \frac{d\lambda}{|\lambda|}\\
&=\int_{\R^*}tr(\pi(-x,-y,0)\pi(2a\lambda k,al,\lambda am)\mathcal{G}(\lambda))\ \overline{(\nu,j)(\lambda)}\ \frac{d\lambda}{|\lambda|}\\
&=\int_{\R^*}e^{2\pi i(\lambda am-\lambda ak\cdot y+\frac{a}{2}l\cdot x)}tr(\pi(2\lambda ak-x,al-y,0)\mathcal{G}(\lambda))\ \overline{(\nu,j)(\lambda)}\ \frac{d\lambda}{|\lambda|}\\
&=e^{\pi ial\cdot x}\int_{\R^*}e^{2\pi i\lambda a(m-k\cdot y)}tr(\pi(2\lambda ak-x,al-y,0)\mathcal{G}(\lambda))\ \overline{(\nu,j)(\lambda)}\ \frac{d\lambda}{|\lambda|}\\
&=e^{\pi ial\cdot x}\int_{\R^*}H_{k,l,m}(z,\lambda)\ \overline{(\nu,j)(\lambda)}\ \frac{d\lambda}{|\lambda|}\ ,
\end{align*}
say. Then $g_{k,l,m}(z,\nu,j)=\mathcal{F}_2H_{k,l,m}(z,\cdot)(\nu,j)$, where $\mathcal{F}_2$ denotes the group Fourier transform on $\R^*$ in the second variable. Now using the fact that the Weyl transform is an isometric isomorphism of $L^2(\R^n)$ onto $\b(L^2(\R^n))$ in \eqref{10}, we get
\begin{align*}
\sum_{k,l\in\Z,m\in\z}\theta_{k,l,m,j,j}^{(1)}&=\sum_{k,l\in\Z,m\in\z}\int_{\R_+^*}|\langle f(\cdot,\nu,j)),g_{k,l,m}(\cdot,\nu,j)\rangle|^2~\frac{d\nu}{\nu}\\
&\leq\sum_{k,l\in\Z,m\in\z}\int_{\R_+^*}\|f(\cdot,\nu,j)\|^2~\|g_{k,l,m}(\cdot,\nu,j)\|^2~\frac{d\nu}{\nu}\\
&=\int_{\R_+^*}\|f(\cdot,\nu,j)\|^2\sum_{k,l\in\Z,m\in\z}\|g_{k,l,m}(\cdot,\nu,j)\|^2~\frac{d\nu}{\nu}\\
&\leq B_1\int_{\R_+^*}\|f(\cdot,\nu,j)\|^2~\frac{d\nu}{\nu}~,\numberthis\label{2}
\end{align*}
by our hypothesis. For $k,l\in\Z,m\in\z$,
\begin{align*}
&\theta_{k,l,m,1,-1}^{(1)}+\theta_{k,l,m,-1,1}^{(1)}\\
&=\int_{\R_+^*}\langle \widehat{F}(\nu,1),(M_{a(2k,l,m)}\mathcal{G})^{\widehat{}}~(\nu,1)\rangle\overline{\langle \widehat{F}(\nu,-1),(M_{a(2k,l,m)}\mathcal{G})^{\widehat{}}~(\nu,-1)\rangle}~\frac{d\nu}{\nu}~+\\
&\int_{\R_+^*}\langle \widehat{F}(\nu,-1),(M_{a(2k,l,m)}\mathcal{G})^{\widehat{}}~(\nu,-1)\rangle\overline{\langle \widehat{F}(\nu,1),(M_{a(2k,l,m)}\mathcal{G})^{\widehat{}}~(\nu,1)\rangle}~\frac{d\nu}{\nu}\\
&=2~\Re\Bigg(\int_{\R_+^*}\langle \widehat{F}(\nu,1),(M_{a(2k,l,m)}\mathcal{G})^{\widehat{}}~(\nu,1)\rangle\overline{\langle \widehat{F}(\nu,-1),(M_{a(2k,l,m)}\mathcal{G})^{\widehat{}}~(\nu,-1)\rangle}~\frac{d\nu}{\nu}\Bigg)\\
&\leq 2~\Bigg|\int_{\R_+^*}\langle \widehat{F}(\nu,1),(M_{a(2k,l,m)}\mathcal{G})^{\widehat{}}~(\nu,1)\rangle\overline{\langle \widehat{F}(\nu,-1),(M_{a(2k,l,m)}\mathcal{G})^{\widehat{}}~(\nu,-1)\rangle}~\frac{d\nu}{\nu}\Bigg|\\
&\leq 2~\int_{\R_+^*}|\langle \widehat{F}(\nu,1),(M_{a(2k,l,m)}\mathcal{G})^{\widehat{}}~(\nu,1)\rangle|~|\langle \widehat{F}(\nu,-1),(M_{a(2k,l,m)}\mathcal{G})^{\widehat{}}~(\nu,-1)\rangle|~\frac{d\nu}{\nu}\\
&\leq\int_{\R_+^*}|\langle \widehat{F}(\nu,1),(M_{a(2k,l,m)}\mathcal{G})^{\widehat{}}~(\nu,1)\rangle|^2~\frac{d\nu}{\nu}+\int_{\R_+^*}|\langle \widehat{F}(\nu,-1),(M_{a(2k,l,m)}\mathcal{G})^{\widehat{}}~(\nu,-1)\rangle|^2~\frac{d\nu}{\nu}\\
&=\int_{\R_+^*}|\langle\widehat{F}(\nu,1),\widehat{\mathcal{G}}_{k,l,m}~(\nu,1)\rangle|^2~\frac{d\nu}{\nu}+\int_{\R_+^*}|\langle\widehat{F}(\nu,-1),\widehat{\mathcal{G}}_{k,l,m}~(\nu,-1)\rangle|^2~\frac{d\nu}{\nu}\\
&=\int_{\R_+^*}|\langle f(\cdot,\nu,1)),g_{k,l,m}(\cdot,\nu,1)\rangle|^2~\frac{d\nu}{\nu}+\int_{\R_+^*}|\langle f(\cdot,\nu,-1)),g_{k,l,m}(\cdot,\nu,-1)\rangle|^2~\frac{d\nu}{\nu}\ ,
\end{align*}
by using $\widehat{F}(\nu,j)=W(f(\cdot,\nu,j)),\ \widehat{\mathcal{G}}_{k,l,m}~(\nu,j)=W(g_{k,l,m}(\cdot,\nu,j))$ and Plancherel formula for the Weyl transform. Thus for $k,l\in\Z,m\in\z$
\begin{align*}
\theta_{k,l,m,1,-1}^{(1)}&+\theta_{k,l,m,-1,1}^{(1)}\\
&\leq\int_{\R_+^*}\|f(\cdot,\nu,1)\|^2~\|g_{k,l,m}(\cdot,\nu,1)\|^2~\frac{d\nu}{\nu}+\int_{\R_+^*}\|f(\cdot,\nu,-1)\|^2~\|g_{k,l,m}(\cdot,\nu,-1)\|^2~\frac{d\nu}{\nu}.\numberthis\label{3}
\end{align*} 
From \eqref{1} we can write
\begin{align*}
b\sum_{p\in\z}|\langle F,T_{e^{bp}}M_{a(2k,l,m)}\mathcal{G}\rangle|^2&=\sum_{j\in\{1,-1\}}\theta_{k,l,m,j,j}^{(1)}~+\theta_{k,l,m,1,-1}^{(1)}+\theta_{k,l,m,-1,1}^{(1)}+\sum_{j,j^\prime\in\{1,-1\}}\theta_{k,l,m,j,j^\prime}^{(2)}.
\end{align*}
Taking summation over all $k,l,m$ and making use of \eqref{2}, we get
\begin{align*}
&b\sum_{k,l\in\Z,m,p\in\z}|\langle F,T_{e^{bp}}M_{a(2k,l,m)}\mathcal{G}\rangle|^2\\
&=\sum_{\substack{k,l\in\Z,m\in\z \\ j\in\{1,-1\}}}\theta_{k,l,m,j,j}^{(1)}~~+\sum_{k,l\in\Z,m\in\z}\theta_{k,l,m,1,-1}^{(1)}+\sum_{k,l\in\Z,m\in\z}\theta_{k,l,m,-1,1}^{(1)}+\sum_{\substack{k,l\in\Z,m\in\z \\ j,j^\prime\in\{1,-1\}}}\theta_{k,l,m,j,j^\prime}^{(2)}\\
&=\sum_{\substack{k,l\in\Z,m\in\z \\ j\in\{1,-1\}}}\theta_{k,l,m,j,j}^{(1)}\ +R\ ,
\end{align*}
say. Then
\begin{align*}
b\sum_{k,l\in\Z,m,p\in\z}|\langle F,T_{e^{bp}}M_{a(2k,l,m)}\mathcal{G}\rangle|^2&\leq B_1\sum_{j\in\{1,-1\}}\int_{\R_+^*}\|f(\cdot,\nu,j)\|^2~\frac{d\nu}{\nu}\ +R\\
&=B_1\sum_{j\in\{1,-1\}}\int_{\R_+^*}\|W(f(\cdot,\nu,j))\|^2~\frac{d\nu}{\nu}\ +R\\
&=B_1\sum_{j\in\{1,-1\}}\int_{\R_+^*}\|\widehat{F}(\nu,j)\|^2~\frac{d\nu}{\nu}\ +R\\
&=B_1\|F\|^2\ +R\\
&=B_1\|F\|^2\ +R.\numberthis\label{4}
\end{align*}
Indeed $R$ is a real number. Now
\begin{align*}
|R|&\leq\sum_{k,l\in\Z,m\in\z}|\theta_{k,l,m,1,-1}^{(1)}+\theta_{k,l,m,-1,1}^{(1)}|\ +\sum_{k,l\in\Z,m\in\z}\sum_{j,j^\prime\in\{1,-1\}}|\theta_{k,l,m,j,j^\prime}^{(2)}|.\numberthis\label{5}
\end{align*}
Taking modulus and then summation over all $k,l,m$ in both sides of \eqref{3}, we have
\begin{align*}
\sum_{k,l\in\Z,m\in\z}|\theta_{k,l,m,1,-1}^{(1)}+\theta_{k,l,m,-1,1}^{(1)}|&\leq\sum_{k,l\in\Z,m\in\z}\sum_{j\in\{1,-1\}}\int_{\R_+^*}\|f(\cdot,\nu,j)\|^2~\|g_{k,l,m}(\cdot,\nu,j)\|^2~\frac{d\nu}{\nu}\\
&=\sum_{j\in\{1,-1\}}\int_{\R_+^*}\|f(\cdot,\nu,j)\|^2\sum_{k,l\in\Z,m\in\z}\|g_{k,l,m}(\cdot,\nu,j)\|^2~\frac{d\nu}{\nu}\\
&\leq B_1\sum_{j\in\{1,-1\}}\int_{\R_+^*}\|f(\cdot,\nu,j)\|^2\\
&=B_1\|F\|^2.\numberthis\label{6}
\end{align*}
Now 
\begin{align*}
|\theta_{k,l,m,j,j^\prime}^{(2)}|&\leq\sum_{s^\prime\in\z\setminus\{0\}}\int_{\R_+^*}|\langle \widehat{F}(\nu,j),(M_{a(2k,l,m)}\mathcal{G})^{\widehat{}}~(\nu,j)\rangle|~|\langle \widehat{F}(\nu e^{s^\prime/b},j^\prime),(M_{a(2k,l,m)}\mathcal{G})^{\widehat{}}~(\nu e^{s^\prime/b},j^\prime)\rangle|~\frac{d\nu}{\nu}\\
&=\sum_{s^\prime\in\z\setminus\{0\}}\int_{\R_+^*}|\langle\widehat{F}(\nu,j),\widehat{\mathcal{G}}_{k,l,m}~(\nu,j)\rangle|\ |\langle\widehat{F}(\nu e^{s^\prime/b},j^\prime),\widehat{\mathcal{G}}_{k,l,m}~(\nu e^{s^\prime/b},j^\prime)\rangle|~\frac{d\nu}{\nu}\\
&=\sum_{s^\prime\in\z\setminus\{0\}}\int_{\R_+^*}\hspace{-2 mm}|\langle W(f(\cdot,\nu,j)),W(g_{k,l,m}(\cdot,\nu,j))\rangle|~|\langle W(f(\cdot,\nu e^{s^\prime/b},j^\prime)),W(g_{k,l,m}(\cdot,\nu e^{s^\prime/b},j^\prime))\rangle|~\frac{d\nu}{\nu}\\
&=\sum_{s^\prime\in\z\setminus\{0\}}\int_{\R_+^*}|\langle f(\cdot,\nu,j),g_{k,l,m}(\cdot,\nu,j)\rangle|~|\langle f(\cdot,\nu e^{s^\prime/b},j^\prime),g_{k,l,m}(\cdot,\nu e^{s^\prime/b},j^\prime)\rangle|~\frac{d\nu}{\nu}\\
&\leq\sum_{s^\prime\in\z\setminus\{0\}}\int_{\R_+^*}\int_{\c^n\times\c^n}|f(z,\nu,j)||f(z^\prime,\nu e^{s^\prime/b},j^\prime)||g_{k,l,m}(z,\nu,j)||g_{k,l,m}(z^\prime,\nu e^{s^\prime/b},j^\prime)|\ dzdz^\prime\frac{d\nu}{\nu}.
\end{align*}
Now considering summation over all $k,l,m$ and then taking the summation over $k,l,m$ inside the integral, we get 
\begin{align*}
\sum_{k,l\in\Z,m\in\z}|\theta_{k,l,m,j,j^\prime}^{(2)}|&\leq\sum_{s^\prime\in\z\setminus\{0\}}\int_{\R_+^*}\int_{\c^n\times\c^n}|f(z,\nu,j)||f(z^\prime,\nu e^{s^\prime/b},j^\prime)|\ H_{s^\prime,j,j^\prime}(z,z^\prime,\nu)\ dzdz^\prime\frac{d\nu}{\nu},
\end{align*} 
where we put
\begin{align*}
H_{s^\prime,j,j^\prime}(z,z^\prime,\nu)&=\sum_{k,l\in\Z,m\in\z}|g_{k,l,m}(z,\nu,j)||g_{k,l,m}(z^\prime,\nu e^{s^\prime/b},j^\prime)|\ \geq 0.
\end{align*}
An application of Fubini's theorem leads to
\begin{align*}
\sum_{k,l\in\Z,m\in\z}&|\theta_{k,l,m,j,j^\prime}^{(2)}|\\
&\leq\sum_{s^\prime\in\z\setminus\{0\}}\int_{\c^n\times\c^n}\int_{\R_+^*}\big(|f(z,\nu,j)|\ H_{s^\prime,j,j^\prime}^{1/2}(z,z^\prime,\nu)\big)\big(|f(z^\prime,\nu e^{s^\prime/b},j^\prime)|\ H_{s^\prime,j,j^\prime}^{1/2}(z,z^\prime,\nu)\big)\ \frac{d\nu}{\nu}dzdz^\prime.
\end{align*}
Now using Cauchy Schwarz inequality, we get
\begin{align*}
\sum_{k,l\in\Z,m\in\z}|\theta_{k,l,m,j,j^\prime}^{(2)}|&\leq\sum_{s^\prime\in\z\setminus\{0\}}\int_{\c^n\times\c^n}\Bigg(\int_{\R_+^*}|f(z,\nu,j)|^2\ H_{s^\prime,j,j^\prime}(z,z^\prime,\nu)\ \frac{d\nu}{\nu}\Bigg)^{1/2}\times\\
&\Bigg(\int_{\R_+^*}|f(z^\prime,\nu e^{s^\prime/b},j^\prime)|^2\ H_{s^\prime,j,j^\prime}(z,z^\prime,\nu)\ \frac{d\nu}{\nu}\Bigg)^{1/2}dzdz^\prime.
\end{align*}
Taking summation over all $j,j^\prime$ and then again applying Cauchy-Schwarz inequality for the summation, we get
\begin{align*}
\sum_{\substack{k,l\in\Z,m\in\z \\ j,j^\prime\in\{1,-1\}}}|\theta_{k,l,m,j,j^\prime}^{(2)}|&\leq\int_{\c^n\times\c^n}\sum_{\substack{ s^\prime\in\z\setminus\{0\} \\ j,j^\prime\in\{1,-1\}}}\Bigg(\int_{\R_+^*}|f(z,\nu,j)|^2\ H_{s^\prime,j,j^\prime}(z,z^\prime,\nu)\ \frac{d\nu}{\nu}\Bigg)^{1/2}\times\\
&\Bigg(\int_{\R_+^*}|f(z^\prime,\nu e^{s^\prime/b},j^\prime)|^2\ H_{s^\prime,j,j^\prime}(z,z^\prime,\nu)\ \frac{d\nu}{\nu}\Bigg)^{1/2}dzdz^\prime\\
&\leq\int_{\c^n\times\c^n}\Bigg(\sum_{\substack{ s^\prime\in\z\setminus\{0\} \\ j,j^\prime\in\{1,-1\}}}\int_{\R_+^*}|f(z,\nu,j)|^2\ H_{s^\prime,j,j^\prime}(z,z^\prime,\nu)\ \frac{d\nu}{\nu}\Bigg)^{1/2}\times\\
&\Bigg(\sum_{\substack{ s^\prime\in\z\setminus\{0\} \\ j,j^\prime\in\{1,-1\}}}\int_{\R_+^*}|f(z^\prime,\nu e^{s^\prime/b},j^\prime)|^2\ H_{s^\prime,j,j^\prime}(z,z^\prime,\nu)\ \frac{d\nu}{\nu}\Bigg)^{1/2}dzdz^\prime.
\end{align*}
Using change of variables in the second product, we get
\begin{align*}
\sum_{\substack{k,l\in\Z,m\in\z \\ j,j^\prime\in\{1,-1\}}}|\theta_{k,l,m,j,j^\prime}^{(2)}|
&\leq\int_{\c^n\times\c^n}\Bigg(\sum_{\substack{ s^\prime\in\z\setminus\{0\} \\ j,j^\prime\in\{1,-1\}}}\int_{\R_+^*}|f(z,\nu,j)|^2\ H_{s^\prime,j,j^\prime}(z,z^\prime,\nu)\ \frac{d\nu}{\nu}\Bigg)^{1/2}\times\\
&\Bigg(\sum_{\substack{ s^\prime\in\z\setminus\{0\} \\ j,j^\prime\in\{1,-1\}}}\int_{\R_+^*}|f(z^\prime,\nu ,j^\prime)|^2\ H_{s^\prime,j,j^\prime}(z,z^\prime,\nu e^{-s^\prime/b})\ \frac{d\nu}{\nu}\Bigg)^{1/2}dzdz^\prime\\
&\leq\int_{\c^n\times\c^n}\Bigg(\sum_{j,j^\prime\in\{1,-1\}}\int_{\R_+^*}|f(z,\nu,j)|^2\ \sum_{s^\prime\in\z\setminus\{0\}}H_{s^\prime,j,j^\prime}(z,z^\prime,\nu)\ \frac{d\nu}{\nu}\Bigg)^{1/2}\times\\
&\Bigg(\sum_{ j,j^\prime\in\{1,-1\}}\int_{\R_+^*}|f(z^\prime,\nu ,j^\prime)|^2\ \sum_{s^\prime\in\z\setminus\{0\}}H_{s^\prime,j^\prime,j}(z^\prime,z,\nu)\ \frac{d\nu}{\nu}\Bigg)^{1/2}dzdz^\prime,\numberthis\label{7}
\end{align*}
by using the pointwise identity $H_{s^\prime,j,j^\prime}(z,z^\prime,\nu e^{-s^\prime/b})=H_{-s^\prime,j^\prime,j}(z^\prime,z,\nu)$. For $j,j^\prime\in\{1,-1\}$ let us consider the functions
\begin{align*}
\beta_{j,j^\prime}(z,z^\prime)=\esssup_{\nu\in\R_+^*}\sum_{s^\prime\in\z\setminus\{0\}}H_{s^\prime,j,j^\prime}(z,z^\prime,\nu),\ z,z^\prime\in\c^n
\end{align*}
and
\begin{align*}
\alpha_j(z,z^\prime)=\sum_{j^{\prime\prime}\in\{1,-1\}}\beta_{j,j^{\prime\prime}}(z,z^\prime),~~z,z^\prime\in\c^n.
\end{align*}
Then \eqref{7} reduces to
\begin{align*}
\sum_{\substack{k,l\in\Z,m\in\z \\ j,j^\prime\in\{1,-1\}}}|\theta_{k,l,m,j,j^\prime}^{(2)}|&\leq\int_{\c^n\times\c^n}\Bigg(\sum_{j,j^\prime\in\{1,-1\}}\beta_{j,j^\prime}(z,z^\prime)\int_{\R_+^*}|f(z,\nu,j)|^2\ \frac{d\nu}{\nu}\Bigg)^{1/2}\times\\
&\Bigg(\sum_{ j,j^\prime\in\{1,-1\}}\beta_{j^\prime,j}(z^\prime,z)\int_{\R_+^*}|f(z^\prime,\nu ,j^\prime)|^2\ \frac{d\nu}{\nu}\Bigg)^{1/2}dzdz^\prime\\
&\leq\int_{\c^n\times\c^n}\Bigg[\sum_{j\in\{1,-1\}}\Bigg(\int_{\R_+^*}|f(z,\nu,j)|^2\ \frac{d\nu}{\nu}\Bigg)\sum_{j^\prime\in\{1,-1\}}\beta_{j,j^\prime}(z,z^\prime)\Bigg]^{1/2}\times\\
&\Bigg[\sum_{ j^\prime\in\{1,-1\}}\Bigg(\int_{\R_+^*}|f(z^\prime,\nu ,j^\prime)|^2\ \frac{d\nu}{\nu}\Bigg)\sum_{j\in\{1,-1\}}\beta_{j^\prime,j}(z^\prime,z)\Bigg]^{1/2}dzdz^\prime\\
&=\int_{\c^n\times\c^n}\Bigg[\sum_{j\in\{1,-1\}}\Bigg(\int_{\R_+^*}|f(z,\nu,j)|^2\ \frac{d\nu}{\nu}\Bigg)\alpha_j(z,z^\prime)\Bigg]^{1/2}\times\\
&\Bigg[\sum_{ j^\prime\in\{1,-1\}}\Bigg(\int_{\R_+^*}|f(z^\prime,\nu ,j^\prime)|^2\ \frac{d\nu}{\nu}\Bigg)\alpha_{j^\prime}(z^\prime,z)\Bigg]^{1/2}dzdz^\prime.
\end{align*}
Now we define $\alpha(z,z^\prime)=\sum_{j\in\{1,-1\}}\alpha_{j}(z,z^\prime)$, for $z,z^\prime\in\c^n$. Then from above equation, we get
\begin{align*}
\sum_{\substack{k,l\in\Z,m\in\z \\ j,j^\prime\in\{1,-1\}}}|\theta_{k,l,m,j,j^\prime}^{(2)}|&\leq\int_{\c^n\times\c^n}\alpha^{1/2}(z,z^\prime)\Bigg(\sum_{j\in\{1,-1\}}\int_{\R_+^*}|f(z,\nu,j)|^2\ \frac{d\nu}{\nu}\Bigg)^{1/2}\times\\
&\alpha^{1/2}(z^\prime,z)\Bigg(\sum_{ j^\prime\in\{1,-1\}}\int_{\R_+^*}|f(z^\prime,\nu ,j^\prime)|^2\ \frac{d\nu}{\nu}\Bigg)^{1/2}dzdz^\prime\\
&=\int_{\c^n\times\c^n}\big(\alpha^{1/2}(z,z^\prime)\|f(z,\cdot)\|\big)\big(\alpha^{1/2}(z^\prime,z)\|f(z^\prime,\cdot)\|\big)\ dzdz^\prime\\
&\leq\Bigg(\int_{\c^n\times\c^n}\alpha(z,z^\prime)\|f(z,\cdot)\|^2\ dzdz^\prime\Bigg)^{1/2}\Bigg(\int_{\c^n\times\c^n}\alpha(z^\prime,z)\|f(z^\prime,\cdot)\|^2\ dzdz^\prime\Bigg)^{1/2}\\
&=\Bigg(\int_{\c^n}\|f(z,\cdot)\|^2\|\alpha(z,\cdot)\|_{L^1(\c^n)}\ dz\Bigg)^{1/2}\Bigg(\int_{\c^n}\|f(z^\prime,\cdot)\|^2\|\alpha(z^\prime,\cdot)\|_{L^1(\c^n)}\ dz^\prime\Bigg)^{1/2}\\
&=\int_{\c^n}\|f(z,\cdot)\|^2\|\alpha(z,\cdot)\|_{L^1(\c^n)}\ dz\\
&\leq B_2\int_{\c^n}\|f(z,\cdot)\|^2\ dz\ ,\numberthis\label{8}
\end{align*}
by our hypothesis. On the other hand using the Weyl transform, we have
\begin{align*}
\int_{\c^n}\|f(z,\cdot)\|^2\ dz&=\int_{\c^n}\sum_{j\in\{1,-1\}}\int_{\R_+^*}|f(z,\nu,j)|^2\ \frac{d\nu}{\nu}dz\\
&=\sum_{j\in\{1,-1\}}\int_{\R_+^*}\int_{\c^n}|f(z,\nu,j)|^2\ dz\frac{d\nu}{\nu}\\
&=\sum_{j\in\{1,-1\}}\int_{\R_+^*}\|f(\cdot,\nu,j)\|_{L^2(\c^n)}^2\ \frac{d\nu}{\nu}\\
&=\sum_{j\in\{1,-1\}}\int_{\R_+^*}\|W(f(\cdot,\nu,j))\|_\b^2\ \frac{d\nu}{\nu}\\
&=\sum_{j\in\{1,-1\}}\int_{\R_+^*}\|\widehat{F}(\nu,j)\|_\b^2\ \frac{d\nu}{\nu}\\
&=\|F\|_{L^2(\R^*,\b;d\kappa)}^2.
\end{align*}
Thus \eqref{8} becomes
\begin{align}\label{9}
\sum_{\substack{k,l\in\Z,m\in\z \\ j,j^\prime\in\{1,-1\}}}|\theta_{k,l,m,j,j^\prime}^{(2)}|&\leq B_2\|F\|_{L^2(\R^*,\b;d\kappa)}^2.
\end{align}
Using \eqref{9} and \eqref{6} in \eqref{5}, we get $|R|\leq(B_1+B_2)\|F\|_{L^2(\R^*,\b;d\kappa)}^2$. Hence \eqref{4} gives $\mathcal{U}(\mathcal{G})$ is a Bessel sequence for $L^2(\R^*,\b;d\kappa)$.
\end{proof}
\section{Necessary and Sufficient Condition for a Gabor Frame Sequence}\label{z}
Let $P(\mathcal{G})=\overline{span}\ \mathcal{U}(\mathcal{G})$. Consider $F\in span\ \mathcal{U}(\mathcal{G})$. Then $F=\sum_{(k,l,m,p)\in\mathcal{F}}\alpha_{k,l,m,p}T_{e^{bp}}M_{a(2k,l,m)}\mathcal{G}$, for some finite subset $\mathcal{F}$ of $\z^{2n+2}$. For $(\nu,j)\in\R_+^*\times\{1,-1\}$,
\begin{align*}
\widehat{F}(\nu,j)&=\sum_{(k,l,m,p)\in\mathcal{F}}(\alpha_{k,l,m,p}T_{e^{bp}}M_{a(2k,l,m)}\mathcal{G})^{\widehat{}}\ (\nu,j)\\
&=\sum_{(k,l,m,p)\in\mathcal{F}}\alpha_{k,l,m,p}\overline{(\nu,j)(e^{bp})}(M_{a(2k,l,m)}\mathcal{G})^{\widehat{}}\ (\nu,j)\\
&=\sum_{(k,l,m,p)\in\mathcal{F}}\alpha_{k,l,m,p}e^{-2\pi ibp\log\nu}\widehat{\mathcal{G}}_{k,l,m}(\nu,j)\\
&=\sum_{(k,l,m)\in\mathcal{F}^\prime}\rho_{k,l,m}(\nu)\widehat{\mathcal{G}}_{k,l,m}(\nu,j)\ ,
\end{align*}
where $\rho_{k,l,m}(\nu)=\sum_{p\in\mathcal{F}^{\prime\prime}}\alpha_{k,l,m,p}e^{-2\pi ibp\log\nu}$ and $\mathcal{F}^\prime,\mathcal{F}^{\prime\prime}$ are corresponding finite subsets of $\z^{2n+1}$ and $\z$ respectively. Again as in section \ref{b} we can write $\widehat{\mathcal{G}}_{k,l,m}$ in terms of the Weyl transform, namely $\widehat{\mathcal{G}}_{k,l,m}(\nu,j)=W(g_{k,l,m}(\cdot,\nu,j))$. Then $\widehat{F}(\nu,j)=W\big(\sum_{(k,l,m)\in\mathcal{F}^\prime}\rho_{k,l,m}(\nu)g_{k,l,m}(\cdot,\nu,j)\big)$. Hence by Plancherel formula, we have
\begin{align*}
\|F\|_{L^2(\R^*,\b;d\kappa)}^2&=\|\widehat{F}\|_{L^2(\R^*,\b;d\kappa)}^2\\
&=\sum_{j\in\{1,-1\}}\int_{\R_+^*}\|\widehat{F}(\nu,j)\|_\b^2\ \frac{d\nu}{\nu}\\
&=\sum_{j\in\{1,-1\}}\int_{\R_+^*}\bigg\|W\bigg(\sum_{(k,l,m)\in\mathcal{F}^\prime}\rho_{k,l,m}(\nu)g_{k,l,m}(\cdot,\nu,j)\bigg)\bigg\|_\b^2\ \frac{d\nu}{\nu}\\
&=\sum_{j\in\{1,-1\}}\int_{\R_+^*}\bigg\|\sum_{(k,l,m)\in\mathcal{F}^\prime}\rho_{k,l,m}(\nu)g_{k,l,m}(\cdot,\nu,j)\bigg\|_{L^2(\c^n)}^2\ \frac{d\nu}{\nu}\ ,
\end{align*}
by using Plancherel formula for the Weyl transform. Now making use of discretization of $\R^*$, we get
\begin{align*}
\|F\|_{L^2(\R^*,\b;d\kappa)}^2&=\sum_{j\in\{1,-1\}}\sum_{s\in\z}\int_{e^{s/b}}^{e^{(s+1)/b}}\bigg\|\sum_{(k,l,m)\in\mathcal{F}^\prime}\rho_{k,l,m}(\nu)g_{k,l,m}(\cdot,\nu,j)\bigg\|_{L^2(\c^n)}^2\ \frac{d\nu}{\nu}\\
&=\sum_{j\in\{1,-1\}}\sum_{s\in\z}\int_1^{e^{1/b}}\bigg\|\sum_{(k,l,m)\in\mathcal{F}^\prime}\rho_{k,l,m}(\nu e^{s/b})g_{k,l,m}(\cdot,\nu e^{s/b},j)\bigg\|_{L^2(\c^n)}^2\ \frac{d\nu}{\nu}\ ,\numberthis\label{13}
\end{align*}
by applying change of variables. But
\begin{align*}
&\bigg\|\sum_{(k,l,m)\in\mathcal{F}^\prime}\rho_{k,l,m}(\nu e^{s/b})g_{k,l,m}(\cdot,\nu e^{s/b},j)\bigg\|_{L^2(\c^n)}^2\\
&=\bigg\langle \sum_{(k,l,m)\in\mathcal{F}^\prime}\rho_{k,l,m}(\nu e^{s/b})g_{k,l,m}(\cdot,\nu e^{s/b},j),\sum_{(k^\prime,l^\prime,m^\prime)\in\mathcal{F}^\prime}\rho_{k^\prime,l^\prime,m^\prime}(\nu e^{s/b})g_{k^\prime,l^\prime,m^\prime}(\cdot,\nu e^{s/b},j)\bigg\rangle_{L^2(\c^n)}\\
&=\sum_{\substack{(k,l,m)\in\mathcal{F}^\prime \\ (k^\prime,l^\prime,m^\prime)\in\mathcal{F}^\prime}}\rho_{k,l,m}(\nu e^{s/b})\overline{\rho_{k^\prime,l^\prime,m^\prime}(\nu e^{s/b})}\ \big\langle g_{k,l,m}(\cdot,\nu e^{s/b},j),g_{k^\prime,l^\prime,m^\prime}(\cdot,\nu e^{s/b},j)\big\rangle_{L^2(\c^n)}\\
&=\sum_{(k,l,m)\in\mathcal{F}^\prime}|\rho_{k,l,m}(\nu e^{s/b})|^2\ \|g_{k,l,m}(\cdot,\nu e^{s/b},j)\|_{L^2(\c^n)}^2\ +\\
&\hspace{1.5 cm}\sum_{(k,l,m)\neq(k^\prime,l^\prime,m^\prime)}\rho_{k,l,m}(\nu e^{s/b})\overline{\rho_{k^\prime,l^\prime,m^\prime}(\nu e^{s/b})}\ \big\langle g_{k,l,m}(\cdot,\nu e^{s/b},j),g_{k^\prime,l^\prime,m^\prime}(\cdot,\nu e^{s/b},j)\big\rangle_{L^2(\c^n)}.
\end{align*}
Hence \eqref{13} becomes
\begin{align*}
\|F\|_{L^2(\R^*,\b;d\kappa)}^2&=\sum_{\substack{j\in\{1,-1\} \\ s\in\z}}\int_1^{e^{1/b}}\sum_{(k,l,m)\in\mathcal{F}^\prime}|\rho_{k,l,m}(\nu e^{s/b})|^2\ \|g_{k,l,m}(\cdot,\nu e^{s/b},j)\|_{L^2(\c^n)}^2\ \frac{d\nu}{\nu}\ +\\
&\hspace{-1.5 cm}\sum_{\substack{j\in\{1,-1\} \\ s\in\z}}\hspace{-1 mm}\int_1^{e^{1/b}}\hspace{-5 mm}\sum_{(k,l,m)\neq(k^\prime,l^\prime,m^\prime)}\hspace{-3 mm}\rho_{k,l,m}(\nu e^{s/b})\overline{\rho_{k^\prime,l^\prime,m^\prime}(\nu e^{s/b})}\ \big\langle g_{k,l,m}(\cdot,\nu e^{s/b},j),g_{k^\prime,l^\prime,m^\prime}(\cdot,\nu e^{s/b},j)\big\rangle_{L^2(\c^n)}.\numberthis\label{14}
\end{align*}
Assume that $\{g_{k,l,m}(\cdot,\nu e^{s/b},j):k,l\in\Z,m\in\z\}$ is an orthogonal system in $L^2(\c^n)$ for $a.e.\ \nu\in(1,e^{1/b})$ and for each $j\in\{1,-1\},s\in\z$. Then \eqref{14} reduces to
\begin{align*}
\|F\|_{L^2(\R^*,\b;d\kappa)}^2&=\sum_{\substack{j\in\{1,-1\} \\ s\in\z}}\int_1^{e^{1/b}}\sum_{(k,l,m)\in\mathcal{F}^\prime}|\rho_{k,l,m}(\nu e^{s/b})|^2\ \|g_{k,l,m}(\cdot,\nu e^{s/b},j)\|_{L^2(\c^n)}^2\ \frac{d\nu}{\nu}.\numberthis\label{15}
\end{align*}
Using the definition of $\rho_{k,l,m}$, we have
\begin{align*}
\rho_{k,l,m}(\nu e^{s/b})&=\sum_{p\in\mathcal{F}^{\prime\prime}}\alpha_{k,l,m,p}e^{-2\pi ibp\log(\nu e^{s/b})}=\sum_{p\in\mathcal{F}^{\prime\prime}}\alpha_{k,l,m,p}e^{-2\pi ibp\log\nu}=\rho_{k,l,m}(\nu).
\end{align*}
Hence \eqref{15} gives
\begin{align*}
\|F\|_{L^2(\R^*,\b;d\kappa)}^2&=\sum_{\substack{j\in\{1,-1\} \\ s\in\z}}\int_1^{e^{1/b}}\sum_{(k,l,m)\in\mathcal{F}^\prime}|\rho_{k,l,m}(\nu)|^2\ \|g_{k,l,m}(\cdot,\nu e^{s/b},j)\|_{L^2(\c^n)}^2\ \frac{d\nu}{\nu}\\
&=\sum_{\substack{j\in\{1,-1\} \\ (k,l,m)\in\mathcal{F}^\prime}}\int_1^{e^{1/b}}|\rho_{k,l,m}(\nu)|^2\ \sum_{s\in\z}\|g_{k,l,m}(\cdot,\nu e^{s/b},j)\|_{L^2(\c^n)}^2\ \frac{d\nu}{\nu}\\
&=\sum_{(k,l,m)\in\mathcal{F}^\prime}\int_1^{e^{1/b}}|\rho_{k,l,m}(\nu)|^2\ \sum_{j\in\{1,-1\}}\sum_{s\in\z}\|g_{k,l,m}(\cdot,\nu e^{s/b},j)\|_{L^2(\c^n)}^2\ \frac{d\nu}{\nu}\\
&=\sum_{(k,l,m)\in\mathcal{F}^\prime}\int_1^{e^{1/b}}|\rho_{k,l,m}(\nu)|^2\ w_{g_{k,l,m}}(\nu)\ \frac{d\nu}{\nu}\ ,
\end{align*}
where
\begin{align*}
w_{g_{k,l,m}}(\nu)&=\sum_{j\in\{1,-1\}}\sum_{s\in\z}\|g_{k,l,m}(\cdot,\nu e^{s/b},j)\|_{L^2(\c^n)}^2.
\end{align*}
Therefore 
\begin{align*}
\|F\|_{L^2(\R^*,\b;d\kappa)}^2&=\sum_{(k,l,m)\in\mathcal{F}^\prime}\|\rho_{k,l,m}w_{g_{k,l,m}}^{1/2}\|_{L^2((1,e^{1/b});d\kappa)}^2=\big\|\rho\big\|_{\ell^2\big(\z^{2n+1},L^2((1,e^{1/b});d\kappa)\big)}^2\ ,\numberthis\label{16}
\end{align*}
where
\begin{align*}
\rho=\big\{\rho_{k,l,m}w_{g_{k,l,m}}^{1/2}\big\}_{(k,l,m)\in\z^{2n+1}}\in\ell^2\big(\z^{2n+1},L^2((1,e^{1/b});d\kappa)\big)\numberthis\label{17}
\end{align*} 
and $\rho_{k,l,m}=0$ for all $(k,l,m)\notin\mathcal{F}^\prime$.\\

We define $c_{00}\big(\z^{2n+1},L^2((1,e^{1/b});d\kappa)\big)$ to be the set of all $L^2((1,e^{1/b});d\kappa)$ valued sequences having atmost finitely many nonzero functions.
\begin{proposition}\label{c}
Let $\{g_{k,l,m}(\cdot,\nu e^{s/b},j):k,l\in\Z,m\in\z\}$ be an orthogonal system in $L^2(\c^n)$ for $a.e.\ \nu\in(1,e^{1/b})$ and for each $j\in\{1,-1\},s\in\z$. Define 
$$w_{g_{k,l,m}}(\nu)=\sum_{j\in\{1,-1\}}\sum_{s\in\z}\|g_{k,l,m}(\cdot,\nu e^{s/b},j)\|_{L^2(\c^n)}^2\ and\ \rho=\big\{\rho_{k,l,m}w_{g_{k,l,m}}^{1/2}\big\}_{(k,l,m)\in\z^{2n+1}},\ where$$
$\rho_{k,l,m}(\nu)=\sum_{p\in\mathcal{F}^{\prime\prime}}\alpha_{k,l,m,p}e^{-2\pi ibp\log\nu}$, $\nu\in(1,e^{1/b})$. Then the map $F\mapsto\rho$ initially defined on span $\mathcal{U}(\mathcal{G})$ into $c_{00}\big(\z^{2n+1},L^2((1,e^{1/b});d\kappa)\big)$ can be extended to an isometric isomorphism of $P(\mathcal{G})$ onto $\ell^2\big(\z^{2n+1},L^2((1,e^{1/b});d\kappa)\big)$.
\end{proposition}
\begin{proof}
The isometry follows from \eqref{16} and the density argument. Conversely, let $\rho=\\ \big\{\rho_{k,l,m}w_{g_{k,l,m}}^{1/2}\big\}_{(k,l,m)\in\mathcal{F}}\in c_{00}\big(\z^{2n+1},L^2((1,e^{1/b}),d\kappa)\big)$ with $\rho_{k,l,m}(\nu)=\sum_{p\in\mathcal{F}^{\prime\prime}}\alpha_{k,l,m,p}e^{-2\pi ibp\log\nu}$, $\nu\in(1,e^{1/b})$, $\mathcal{F}$ and $\mathcal{F}^{\prime\prime}$ are finite subsets of $\z^{2n+1}$ and $\z$ respectively. Define $F=\\ \sum_{(k,l,m)\in\mathcal{F},p\in\mathcal{F}^{\prime\prime}}\alpha_{k,l,m,p}T_{e^{bp}}M_{a(2k,l,m)}\mathcal{G}$. Then $F\in span\ \mathcal{U}(\mathcal{G})$ with $\|F\|=\|\rho\|$, by \eqref{16}. Since $span\ \mathcal{U}(\mathcal{G})$ and $c_{00}\big(\z^{2n+1},L^2((1,e^{1/b}),d\kappa)\big)$ are dense in $P(\mathcal{G})$ and $\ell^2\big(\z^{2n+1},L^2((1,e^{1/b}),d\kappa)\big)$ respectively, we get the required result. 
\end{proof}
\begin{theorem}
Let $\{g_{k,l,m}(\cdot,\nu e^{s/b},j):k,l\in\Z,m\in\z\}$ be an orthogonal system in $L^2(\c^n)$ for $a.e.\ \nu\in(1,e^{1/b})$ and for each $j\in\{1,-1\},s\in\z$. Define
$$w_{g_{k,l,m}}(\nu)=\sum_{j\in\{1,-1\}}\sum_{s\in\z}\|g_{k,l,m}(\cdot,\nu e^{s/b},j)\|_{L^2(\c^n)}^2,\  \nu\in(1,e^{1/b}).$$
Then the Gabor system $\mathcal{U}(\mathcal{G})$ is an orthonormal system in $L^2(\R^*,\b;d\kappa)$ iff $w_{g_{k,l,m}}(\nu)=b$ for $a.e.\ \nu\in(1,e^{1/b})$ for each $(k,l,m)\in\z^{2n+1}$.
\end{theorem}
\begin{proof}
Assume that $w_{g_{k,l,m}}(\nu)=b$ for $a.e.\ \nu\in(1,e^{1/b})$ for each $(k,l,m)\in\z^{2n+1}$. Then using Plancherel formula for $L^2(\R^*,\b;d\kappa)$, we get
\begin{align*}
&\big\langle T_{e^{bp}}M_{a(2k,l,m)}\mathcal{G},T_{e^{bp^\prime}}M_{a(2k^\prime,l^\prime,m^\prime)}\mathcal{G}\big\rangle_{L^2(\R^*,\b;d\kappa)}\\
&=\big\langle T_{e^{bp}}\mathcal{G}_{k,l,m},T_{e^{bp^\prime}}\mathcal{G}_{k^\prime,l^\prime,m^\prime}\big\rangle_{L^2(\R^*,\b;d\kappa)}\\
&=\big\langle (T_{e^{bp}}\mathcal{G}_{k,l,m})^{\widehat{}}\ ,(T_{e^{bp^\prime}}\mathcal{G}_{k^\prime,l^\prime,m^\prime})^{\widehat{}}\ \big\rangle_{L^2(\R^*,\b;d\kappa)}\\
&=\sum_{j\in\{1,-1\}}\int_{\R_+^*}\big\langle (T_{e^{bp}}\mathcal{G}_{k,l,m})^{\widehat{}}\ (\nu,j),(T_{e^{bp^\prime}}\mathcal{G}_{k^\prime,l^\prime,m^\prime})^{\widehat{}}\ (\nu,j)\big\rangle_\b\ \frac{d\nu}{\nu}\\
&=\sum_{j\in\{1,-1\}}\int_{\R_+^*}\big\langle\widehat{\mathcal{G}}_{k,l,m}(\nu,j),\widehat{\mathcal{G}}_{k^\prime,l^\prime,m^\prime}(\nu,j)\big\rangle_\b e^{2\pi ib(p^\prime-p)\log\nu}\ \frac{d\nu}{\nu}\\
&=\sum_{j\in\{1,-1\}}\sum_{s\in\z}\int_{e^{s/b}}^{e^{(s+1)/b}}\big\langle\widehat{\mathcal{G}}_{k,l,m}(\nu,j),\widehat{\mathcal{G}}_{k^\prime,l^\prime,m^\prime}(\nu,j)\big\rangle_\b e^{2\pi ib(p^\prime-p)\log\nu}\ \frac{d\nu}{\nu}\\
&=\sum_{j\in\{1,-1\}}\sum_{s\in\z}\int_{e^{s/b}}^{e^{(s+1)/b}}\big\langle g_{k,l,m}(\cdot,\nu,j),g_{k^\prime,l^\prime,m^\prime}(\cdot,\nu,j)\big\rangle_{L^2(\c^n)} e^{2\pi ib(p^\prime-p)\log\nu}\ \frac{d\nu}{\nu}.
\end{align*}
By applying change of variables and then by orthogonality condition, we get
\begin{align*}
&\big\langle T_{e^{bp}}M_{a(2k,l,m)}\mathcal{G},T_{e^{bp^\prime}}M_{a(2k^\prime,l^\prime,m^\prime)}\mathcal{G}\big\rangle_{L^2(\R^*,\b;d\kappa)}\\
&=\sum_{j\in\{1,-1\}}\sum_{s\in\z}\int_1^{e^{1/b}}\big\langle g_{k,l,m}(\cdot,\nu e^{s/b},j),g_{k^\prime,l^\prime,m^\prime}(\cdot,\nu e^{s/b},j)\big\rangle_{L^2(\c^n)} e^{2\pi ib(p^\prime-p)\log(\nu e^{s/b})}\ \frac{d\nu}{\nu}\\
&=\begin{cases}
0, & (k,l,m)\neq(k^\prime,l^\prime,m^\prime) \\
\sum_{j\in\{1,-1\}}\sum_{s\in\z}\int_1^{e^{1/b}}\|g_{k,l,m}(\cdot,\nu e^{s/b},j)\|_{L^2(\c^n)}^2e^{2\pi ib(p^\prime-p)\log\nu}\ \frac{d\nu}{\nu}, & (k,l,m)=(k^\prime,l^\prime,m^\prime)
\end{cases}\\
&=\begin{cases}
0, & (k,l,m)\neq(k^\prime,l^\prime,m^\prime) \\
\int_1^{e^{1/b}}e^{2\pi ib(p^\prime-p)\log\nu} w_{g_{k,l,m}}(\nu)\ \frac{d\nu}{\nu}, & (k,l,m)=(k^\prime,l^\prime,m^\prime).
\end{cases}\numberthis\label{19}
\end{align*}
Thus by the hypothesis we get
\begin{align*}
\big\langle T_{e^{bp}}M_{a(2k,l,m)}\mathcal{G},T_{e^{bp^\prime}}M_{a(2k,l,m)}\mathcal{G}\big\rangle_{L^2(\R^*,\b;d\kappa)}&=b\int_1^{e^{1/b}}e^{2\pi ib(p^\prime-p)\log\nu}\ \frac{d\nu}{\nu}\\
&=b\cdot\frac{1}{b}\delta_{p,p^\prime}\\
&=\delta_{p,p^\prime}\ ,
\end{align*}
proving $\mathcal{U}(\mathcal{G})$ is an orthonormal system.\\

Conversely suppose $\mathcal{U}(\mathcal{G})$ is an orthonormal system in $L^2(\R^*,\b;d\kappa)$. Then from \eqref{19} we have
\begin{align*}
\big\langle M_{a(2k,l,m)}\mathcal{G},T_{e^{bp}}M_{a(2k,l,m)}\mathcal{G}\big\rangle_{L^2(\R^*,\b;d\kappa)}&=\int_1^{e^{1/b}}e^{2\pi ibp\log\nu} w_{g_{k,l,m}}(\nu)\ \frac{d\nu}{\nu}\\
&=\int_0^{1/b}w_{g_{k,l,m}}(e^\nu)e^{2\pi ibp\nu}\ d\nu\\
&=\frac{1}{b}\widehat{w}_{g_{k,l,m}}(e^{-p})\ ,
\end{align*}
by applying change of variables. Hence by hypothesis, we get $\widehat{w}_{g_{k,l,m}}(e^{-p})=b\delta_{p,0}$. By expanding $w_{g_{k,l,m}}$ in the Fourier series, we get
\begin{align*}
w_{g_{k,l,m}}(e^{\nu})=\sum_{p\in\z}\widehat{w}_{g_{k,l,m}}(e^p)e^{2\pi ibp\nu}=\sum_{p\in\z}b\delta_{-p,0}e^{2\pi ibp\nu}=b\ \mathrm{for}\ a.e.\ \nu\in(0,1/b),
\end{align*}
from which it follows that $w_{g_{k,l,m}}(\nu)=b$ for $a.e.\ \nu\in(1,e^{1/b})$.
\end{proof}
\begin{theorem}\label{a}
Let $\{g_{k,l,m}(\cdot,\nu e^{s/b},j):k,l\in\Z,m\in\z\}$ be an orthogonal system in $L^2(\c^n)$ for $a.e.\ \nu\in(1,e^{1/b})$ and for each $j\in\{1,-1\},s\in\z$ and define $\Omega_{k,l,m}=\{\nu\in(1,e^{1/b}):w_{g_{k,l,m}}(\nu)\neq 0\}$ for each $(k,l,m)\in\z^{2n+1}$, where
$$w_{g_{k,l,m}}(\nu)=\sum_{j\in\{1,-1\}}\sum_{s\in\z}\|g_{k,l,m}(\cdot,\nu e^{s/b},j)\|_{L^2(\c^n)}^2,\  \nu\in(1,e^{1/b}).$$
Then the collection $\mathcal{U}(\mathcal{G})$ is a Parseval frame sequence iff $w_{g_{k,l,m}}(\nu)=b$ for $a.e.\ \nu\in\Omega_{k,l,m}$ for each $(k,l,m)\in\z^{2n+1}$.
\end{theorem}
\begin{proof}
Let $\mathcal{U}(\mathcal{G})$ be a Parseval frame sequence. Let $F\in span\ \mathcal{U}(\mathcal{G})$, with 
\begin{align*}
F=\sum_{(k,l,m,p)\in\mathcal{F}}\alpha_{k,l,m,p}T_{e^{bp}}M_{a(2k,l,m)}\mathcal{G}\ ,
\end{align*}
for some finite subset $\mathcal{F}$ of $\z^{2n+2}$. Then as in the discussion before Proposition \ref{c} we can get \eqref{16} and \eqref{17} with
\begin{align*}
\rho_{k,l,m}(\nu)=\sum_{p\in\mathcal{F}^{\prime\prime}}\alpha_{k,l,m,p}e^{-2\pi ibp\log\nu}.
\end{align*}
Now using Plancherel formula for $L^2(\R^*,\b;d\kappa)$, we get
\begin{align*}
&\big\langle F,T_{e^{bp}}M_{a(2k,l,m)}\mathcal{G}\big\rangle_{L^2(\R^*,\b;d\kappa)}\\
&=\big\langle F,T_{e^{bp}}\mathcal{G}_{k,l,m}\big\rangle_{L^2(\R^*,\b;d\kappa)}\\
&=\bigg\langle\sum_{(k^\prime,l^\prime,m^\prime,p^\prime)\in\mathcal{F}}\alpha_{k^\prime,l^\prime,m^\prime,p^\prime}T_{e^{bp^\prime}}\mathcal{G}_{k^\prime,l^\prime,m^\prime},T_{e^{bp}}\mathcal{G}_{k,l,m}\bigg\rangle_{L^2(\R^*,\b;d\kappa)}\\
&=\sum_{\mathcal{F}}\alpha_{k^\prime,l^\prime,m^\prime,p^\prime}\big\langle(T_{e^{bp^\prime}}\mathcal{G}_{k^\prime,l^\prime,m^\prime})^{\widehat{}}\ ,(T_{e^{bp}}\mathcal{G}_{k,l,m})^{\widehat{}}\ \big\rangle_{L^2(\R^*,\b;d\kappa)}\\
&=\sum_{\mathcal{F}}\alpha_{k^\prime,l^\prime,m^\prime,p^\prime}\sum_{j\in\{1,-1\}}\int_{\R_+^*}\big\langle(T_{e^{bp^\prime}}\mathcal{G}_{k^\prime,l^\prime,m^\prime})^{\widehat{}}\ (\nu,j),(T_{e^{bp}}\mathcal{G}_{k,l,m})^{\widehat{}}\ (\nu,j)\big\rangle_\b\ \frac{d\nu}{\nu}\\
&=\sum_{(k^\prime,l^\prime,m^\prime)\in\mathcal{F}^\prime}\sum_{j\in\{1,-1\}}\sum_{s\in\z}\int_{e^{s/b}}^{e^{(s+1)/b}}\rho_{k^\prime,l^\prime,m^\prime}(\nu)\big\langle\widehat{\mathcal{G}}_{k^\prime,l^\prime,m^\prime}(\nu,j),\widehat{\mathcal{G}}_{k,l,m}(\nu,j)\big\rangle_\b e^{2\pi ibp\log\nu}\ \frac{d\nu}{\nu}\\
&=\sum_{(k^\prime,l^\prime,m^\prime)\in\mathcal{F}^\prime}\sum_{j\in\{1,-1\}}\sum_{s\in\z}\int_1^{e^{1/b}}\rho_{k^\prime,l^\prime,m^\prime}(\nu e^{s/b})\big\langle\widehat{\mathcal{G}}_{k^\prime,l^\prime,m^\prime}(\nu e^{s/b},j),\widehat{\mathcal{G}}_{k,l,m}(\nu e^{s/b},j)\big\rangle_\b e^{2\pi ibp\log(\nu e^{s/b})}\ \frac{d\nu}{\nu}\ ,
\end{align*}
by applying change of variables. Hence using the Weyl transform, we get
\begin{align*}
&\big\langle F,T_{e^{bp}}M_{a(2k,l,m)}\mathcal{G}\big\rangle_{L^2(\R^*,\b,d\kappa)}\\
&=\sum_{\mathcal{F}^\prime}\sum_{j\in\{1,-1\}}\sum_{s\in\z}\int_1^{e^{1/b}}\hspace{-6 mm}\rho_{k^\prime,l^\prime,m^\prime}(\nu e^{s/b})\big\langle W(g_{k^\prime,l^\prime,m^\prime}(\cdot,\nu e^{s/b},j)),W(g_{k,l,m}(\cdot,\nu e^{s/b},j))\big\rangle_\b e^{2\pi ibp\log(\nu e^{s/b})}\ \frac{d\nu}{\nu}\\
&=\sum_{\mathcal{F}^\prime}\sum_{j\in\{1,-1\}}\sum_{s\in\z}\int_1^{e^{1/b}}\rho_{k^\prime,l^\prime,m^\prime}(\nu)\big\langle g_{k^\prime,l^\prime,m^\prime}(\cdot,\nu e^{s/b},j),g_{k,l,m}(\cdot,\nu e^{s/b},j)\big\rangle_{L^2(\c^n)}e^{2\pi ibp\log\nu}\ \frac{d\nu}{\nu}\\
&=\int_1^{e^{1/b}}\rho_{k,l,m}(\nu)\ w_{g_{k,l,m}}(\nu)e^{2\pi ibp\log\nu}\ \frac{d\nu}{\nu}\ ,
\end{align*}
where we have made use of the orthogonality condition. Thus
\begin{align*}
\big\langle F,T_{e^{bp}}M_{a(2k,l,m)}\mathcal{G}\big\rangle_{L^2(\R^*,\b;d\kappa)}&=\int_1^{e^{1/b}}\Lambda_{k,l,m}(\nu)e^{2\pi ibp\log\nu}\ \frac{d\nu}{\nu}\\
&=\frac{1}{b}\widehat{\Lambda}_{k,l,m}(e^{-p})\ ,
\end{align*}
where $\Lambda_{k,l,m}(\nu)=\rho_{k,l,m}(\nu)\ w_{g_{k,l,m}}(\nu)$. Hence by using Plancherel formula for $L^2((1,e^{1/b});d\kappa)$,
\begin{align*}
\sum_{p\in\z}\big|\big\langle F,T_{e^{bp}}M_{a(2k,l,m)}\mathcal{G}\big\rangle\big|^2&=\frac{1}{b^2}\sum_{p\in\z}\big|\widehat{\Lambda}_{k,l,m}(e^{-p})\big|^2\\
&=\frac{1}{b}\int_1^{e^{1/b}}\big|\Lambda_{k,l,m}(\nu)\big|^2\ \frac{d\nu}{\nu}\\
&=\frac{1}{b}\int_1^{e^{1/b}}\big|\rho_{k,l,m}(\nu)\big|^2w_{g_{k,l,m}}^2(\nu)\ \frac{d\nu}{\nu}.\numberthis\label{20}
\end{align*}
By our assumption,
\begin{align*}
\sum_{k,l\in\Z,m,p\in\z}\big|\big\langle F,T_{e^{bp}}M_{a(2k,l,m)}\mathcal{G}\big\rangle\big|^2=\|F\|_{L^2(\R^*,\b;d\kappa)}^2.
\end{align*}
Now using \eqref{16} and \eqref{20} in the above equation, we get
\begin{align*}
\frac{1}{b}\sum_{k,l\in\Z,m\in\z}\int_1^{e^{1/b}}\big|\rho_{k,l,m}(\nu)\big|^2w_{g_{k,l,m}}^2(\nu)\ \frac{d\nu}{\nu}=\sum_{k,l\in\Z,m\in\z}\int_1^{e^{1/b}}\big|\rho_{k,l,m}(\nu)\big|^2w_{g_{k,l,m}}(\nu)\ \frac{d\nu}{\nu}.
\end{align*}
This leads to
\begin{align*}
\sum_{k,l\in\Z,m\in\z}\int_{\Omega_{k,l,m}}\big|\rho_{k,l,m}(\nu)\big|^2\big(w_{g_{k,l,m}}(\nu)-b\big)w_{g_{k,l,m}}(\nu)\ \frac{d\nu}{\nu}=0.\numberthis\label{21}
\end{align*}
Consider $\Omega_{k,l,m}^{(1)}=\{\nu\in\Omega_{k,l,m}:w_{g_{k,l,m}}(\nu)\leq b\}$ and $\Omega_{k,l,m}^{(2)}=\{\nu\in\Omega_{k,l,m}:w_{g_{k,l,m}}(\nu)> b\}$ for each $(k,l,m)\in\z^{2n+1}$. Define $\rho^{(1)}=\big\{\rho_{k,l,m}^{(1)}w_{g_{k,l,m}}^{1/2}\big\}_{(k,l,m)\in\z^{2n+1}}$ by
\[\rho_{k,l,m}^{(1)}(\nu)=
\begin{cases}
\frac{1}{\beta_{k,l,m}\|\mathcal{G}_{k,l,m}\|_{L^2(\R^*,\b;d\kappa)}}, & \nu\in\Omega_{k,l,m}^{(1)}\\
0, & otherwise\ ,
\end{cases}
\]
where $\sum_{(k,l,m)\in\z^{2n+1}}\frac{1}{\beta_{k,l,m}^2}<\infty$. We aim to show that $\rho^{(1)}\in\ell^2\big(\z^{2n+1},L^2((1,e^{1/b});d\kappa)\big)$. Consider
\begin{align*}
\big\|\rho^{(1)}\big\|_{\ell^2\big(\z^{2n+1},L^2((1,e^{1/b});d\kappa)\big)}^2&=\sum_{(k,l,m)\in\z^{2n+1}}\big\|\rho_{k,l,m}^{(1)}w_{g_{k,l,m}}^{1/2}\big\|_{L^2((1,e^{1/b});d\kappa)}^2\\
&=\sum_{(k,l,m)\in\z^{2n+1}}\int_1^{e^{1/b}}\big|\rho_{k,l,m}^{(1)}(\nu)\big|^2w_{g_{k,l,m}}(\nu)\ \frac{d\nu}{\nu}\\
&=\sum_{(k,l,m)\in\z^{2n+1}}\frac{1}{\beta_{k,l,m}^2\|\mathcal{G}_{k,l,m}\|_{L^2(\R^*,\b;d\kappa)}^2}\int_{\Omega_{k,l,m}^{(1)}}w_{g_{k,l,m}}(\nu)\ \frac{d\nu}{\nu}\\
&\leq\sum_{(k,l,m)\in\z^{2n+1}}\frac{1}{\beta_{k,l,m}^2\|\mathcal{G}_{k,l,m}\|_{L^2(\R^*,\b;d\kappa)}^2}\int_1^{e^{1/b}}w_{g_{k,l,m}}(\nu)\ \frac{d\nu}{\nu}\numberthis\label{22}
\end{align*}
But
\begin{align*}
\int_1^{e^{1/b}}w_{g_{k,l,m}}(\nu)\ \frac{d\nu}{\nu}&=\int_1^{e^{1/b}}\sum_{j\in\{1,-1\}}\sum_{s\in\z}\|g_{k,l,m}(\cdot,\nu e^{s/b},j)\|_{L^2(\c^n)}^2\ \frac{d\nu}{\nu}\\
&=\sum_{j\in\{1,-1\}}\int_{\R_+^*}\|g_{k,l,m}(\cdot,\nu,j)\|_{L^2(\c^n)}^2\ \frac{d\nu}{\nu}\\
&=\sum_{j\in\{1,-1\}}\int_{\R_+^*}\|W(g_{k,l,m}(\cdot,\nu,j))\|_\b^2\ \frac{d\nu}{\nu}\\
&=\sum_{j\in\{1,-1\}}\int_{\R_+^*}\|\widehat{\mathcal{G}}_{k,l,m}(\nu,j)\|_\b^2\ \frac{d\nu}{\nu}\\
&=\big\|\mathcal{G}_{k,l,m}\big\|_{L^2(\R^*,\b,d\kappa)}^2\ ,
\end{align*}
by applying change of variables and Placherel formula for the Weyl transform. Hence it follows from \eqref{22} that
\begin{align*}
\big\|\rho^{(1)}\big\|^2&\leq\sum_{(k,l,m)\in\z^{2n+1}}\frac{1}{\beta_{k,l,m}^2\|\mathcal{G}_{k,l,m}\|_{L^2(\R^*,\b;d\kappa)}^2}\big\|\mathcal{G}_{k,l,m}\big\|_{L^2(\R^*,\b;d\kappa)}^2=\sum_{(k,l,m)\in\z^{2n+1}}\frac{1}{\beta_{k,l,m}^2}<\infty.
\end{align*}
Thus \eqref{21} is true for $\rho=\rho^{(1)}$. Substituting $\rho=\rho^{(1)}$ in \eqref{21}, we get
\begin{align*}
\sum_{k,l\in\Z,m\in\z}\int_{\Omega_{k,l,m}^{(1)}}\frac{1}{\beta_{k,l,m}^2\|\mathcal{G}_{k,l,m}\|_{L^2(\R^*,\b;d\kappa)}^2}\big(b-w_{g_{k,l,m}}(\nu)\big)w_{g_{k,l,m}}(\nu)\ \frac{d\nu}{\nu}=0
\end{align*}
Hence $\forall\ (k,l,m)\in\z^{2n+1}$,
\begin{align*}
\int_{\Omega_{k,l,m}^{(1)}}\big(b-w_{g_{k,l,m}}(\nu)\big)w_{g_{k,l,m}}(\nu)\ \frac{d\nu}{\nu}=0.
\end{align*}
But on $\Omega_{k,l,m}^{(1)}$ we have $b-w_{g_{k,l,m}}(\nu)\geq 0$ and $w_{g_{k,l,m}}(\nu)>0$. Hence $w_{g_{k,l,m}}(\nu)=b$ for $a.e.\ \nu\in\Omega_{k,l,m}^{(1)}$. Similarly considering $\rho^{(2)}=\big\{\rho_{k,l,m}^{(2)}w_{g_{k,l,m}}^{1/2}\big\}_{(k,l,m)\in\z^{2n+1}}$, where
\[\rho_{k,l,m}^{(2)}(\nu)=
\begin{cases}
\frac{1}{\beta_{k,l,m}\|\mathcal{G}_{k,l,m}\|_{L^2(\R^*,\b;d\kappa)}}, & \nu\in\Omega_{k,l,m}^{(2)}\\
0, & otherwise\ ,
\end{cases}
\]
we can show that $w_{g_{k,l,m}}(\nu)=b$ for $a.e.\ \nu\in\Omega_{k,l,m}^{(2)}$. Therefore $w_{g_{k,l,m}}(\nu)=b$ for $a.e.\ \nu\in\Omega_{k,l,m}\ \mathrm{and}\ \forall\ (k,l,m)\in\z^{2n+1}$.\\

Conversely let $w_{g_{k,l,m}}(\nu)=b$ for $a.e.\ \nu\in\Omega_{k,l,m}$ and for each $(k,l,m)\in\z^{2n+1}$. Then for $F\in span\ \mathcal{U}(G)$ from \eqref{20} and \eqref{16} we have 
\begin{align*}
\sum_{(k,l,m,p)\in\mathcal{F}}\big|\big\langle F,T_{e^{bp}}M_{a(2k,l,m)}\mathcal{G}\big\rangle\big|^2&=\frac{1}{b}\sum_{(k,l,m)\in\mathcal{F}^\prime}\int_1^{e^{1/b}}\big|\rho_{k,l,m}(\nu)\big|^2w_{g_{k,l,m}}^2(\nu)\ \frac{d\nu}{\nu}\\
&=b\sum_{(k,l,m)\in\mathcal{F}^\prime}\int_{\Omega_{k,l,m}}\big|\rho_{k,l,m}(\nu)\big|^2\ \frac{d\nu}{\nu}\\
&=\sum_{(k,l,m)\in\mathcal{F}^\prime}\int_{\Omega_{k,l,m}}\big|\rho_{k,l,m}(\nu)\big|^2w_{g_{k,l,m}}(\nu)\ \frac{d\nu}{\nu}\\
&=\sum_{(k,l,m)\in\mathcal{F}^\prime}\int_1^{e^{1/b}}\big|\rho_{k,l,m}(\nu)\big|^2w_{g_{k,l,m}}(\nu)\ \frac{d\nu}{\nu}\\
&=\|F\|_{L^2(\R^*,\b;d\kappa)}^2\ ,
\end{align*}
proving that $\mathcal{U}(\mathcal{G})$ is a Parseval frame sequence by density argument.      
\end{proof}
\begin{theorem}
Let $\{g_{k,l,m}(\cdot,\nu e^{s/b},j):k,l\in\Z,m\in\z\}$ be an orthogonal system in $L^2(\c^n)$ for $a.e.\ \nu\in(1,e^{1/b})$ and for each $j\in\{1,-1\},s\in\z$. Define 
$$w_{g_{k,l,m}}(\nu)=\sum_{j\in\{1,-1\}}\sum_{s\in\z}\|g_{k,l,m}(\cdot,\nu e^{s/b},j)\|_{L^2(\c^n)}^2,\  \nu\in(1,e^{1/b}).$$ 
Then the family $\mathcal{U}(\mathcal{G})$ is a frame sequence with bounds $A,B>0$ iff $Ab\leq w_{g_{k,l,m}}(\nu)\leq Bb$ for $a.e.\ \nu\in\Omega_{k,l,m},\ \forall\ (k,l,m)\in\z^{2n+1}$. 
\end{theorem}
\begin{proof}
Let $\mathcal{U}(\mathcal{G})$ be a frame sequence with bounds $A,B>0$. Then
\begin{align*}
A\big\|F\big\|_{L^2(\R^*,\b;d\kappa)}^2\leq\sum_{(k,l,m,p)\in\z^{2n+2}}|\langle F,T_{e^{bp}}M_{a(2k,l,m)}\mathcal{G}\rangle|^2\leq B\big\|F\big\|_{L^2(\R^*,\b;d\kappa)}^2\ ,
\end{align*}
for all $F\in P(\mathcal{G})$. In other words by \eqref{16} and \eqref{20}
\begin{align*}
A\|\rho\|^2\leq\frac{1}{b}\sum_{(k,l,m)\in\z^{2n+1}}\int_1^{e^{1/b}}\big|\rho_{k,l,m}(\nu)\big|^2w_{g_{k,l,m}}^2(\nu)\ \frac{d\nu}{\nu}\leq B\|\rho\|^2\ ,\numberthis\label{23}
\end{align*}
where $\rho$ is given by \eqref{17}. From the left hand side inequality of \eqref{23} we have
\begin{align*}
\sum_{(k,l,m)\in\z^{2n+1}}\int_{\Omega_{k,l,m}}\big(Ab-w_{g_{k,l,m}}(\nu)\big)\ \big|\rho_{k,l,m}(\nu)\big|^2w_{g_{k,l,m}}(\nu)\ \frac{d\nu}{\nu}\leq 0.\numberthis\label{24}
\end{align*}
Let $M_{k,l,m}^{(1)}=\{\nu\in\Omega_{k,l,m}:w_{g_{k,l,m}}(\nu)<Ab\}$. We aim to show that $M_{k,l,m}^{(1)}$ has measure zero. We consider $\rho^{(1)}=\big\{\rho_{k,l,m}^{(1)}w_{g_{k,l,m}}^{1/2}\big\}_{(k,l,m)\in\z^{2n+1}}$ such that $\rho_{k,l,m}^{(1)}=\frac{1}{\beta_{k,l,m}\|\mathcal{G}_{k,l,m}\|_{L^2(\R^*,\b;d\kappa)}}\chi_{M_{k,l,m}^{(1)}}$ with $\sum_{(k,l,m)\in\z^{2n+1}}\frac{1}{\beta_{k,l,m}^2}<\infty$. A similar calculation as in Theorem \ref{a} leads to $\rho^{(1)}\in\ell^2\big(\z^{2n+1},\\L^2((1,e^{1/b});d\kappa)\big)$. Substituting $\rho=\rho^{(1)}$ in \eqref{24} we get
\begin{align*}
\sum_{(k,l,m)\in\z^{2n+1}}\frac{1}{\beta_{k,l,m}^2\|\mathcal{G}_{k,l,m}\|_{L^2(\R^*,\b;d\kappa)}^2}\int_{M_{k,l,m}^{(1)}}\big(Ab-w_{g_{k,l,m}}(\nu)\big)w_{g_{k,l,m}}(\nu)\ \frac{d\nu}{\nu}\leq 0.
\end{align*}
But on ${M_{k,l,m}^{(1)}}$ both $Ab-w_{g_{k,l,m}}(\nu)$ and $w_{g_{k,l,m}}(\nu)$ are strictly positive. Therefore $\forall\ (k,l,m)\in\z^{2n+1}$,
\begin{align*}
\int_{M_{k,l,m}^{(1)}}\big(Ab-w_{g_{k,l,m}}(\nu)\big)w_{g_{k,l,m}}(\nu)\ \frac{d\nu}{\nu}=0.
\end{align*}
Since both $Ab-w_{g_{k,l,m}}(\nu)$ and $w_{g_{k,l,m}}(\nu)$ are strictly positive on $M_{k,l,m}^{(1)}$, it follows that $|M_{k,l,m}^{(1)}|=0$. Therefore $Ab\leq w_{g_{k,l,m}}(\nu)$ for $a.e.\ \nu\in\Omega_{k,l,m},\ \forall\ (k,l,m)\in\z^{2n+1}$. From the right hand side inequality of \eqref{23}, we have
\begin{align*}
\sum_{(k,l,m)\in\z^{2n+1}}\int_{\Omega_{k,l,m}}\big(w_{g_{k,l,m}}(\nu)-Bb\big)\ \big|\rho_{k,l,m}(\nu)\big|^2w_{g_{k,l,m}}(\nu)\ \frac{d\nu}{\nu}\leq 0.
\end{align*}
In a similar way substituting $\rho=\rho^{(2)}=\big\{\rho_{k,l,m}^{(2)}w_{g_{k,l,m}}^{1/2}\big\}$ in the above equation we get $w_{g_{k,l,m}}(\nu)\leq Bb$ for $a.e.\ \nu\in\Omega_{k,l,m},\ \forall\ (k,l,m)\in\z^{2n+1}$, where $\rho_{k,l,m}^{(2)}=\frac{1}{\beta_{k,l,m}\|\mathcal{G}_{k,l,m}\|_{L^2(\R^*,\b;d\kappa)}}\chi_{M_{k,l,m}^{(2)}}$ by taking $M_{k,l,m}^{(2)}=\{\nu\in\Omega_{k,l,m}:w_{g_{k,l,m}}(\nu)>Bb\}$.\\ 

Conversely assume that $Ab\leq w_{g_{k,l,m}}(\nu)\leq Bb$ for $a.e.\ \nu\in\Omega_{k,l,m},\ \forall\ (k,l,m)\in\z^{2n+1}$. Then for $F\in span\ \mathcal{U}(\mathcal{G})$ we have \eqref{20}. By our hypothesis
\begin{align*}
A\int_1^{e^{1/b}}\hspace{-3 mm}\big|\rho_{k,l,m}(\nu)\big|^2w_{g_{k,l,m}}(\nu)\ \frac{d\nu}{\nu}\leq\frac{1}{b}\int_1^{e^{1/b}}\hspace{-3 mm}\big|\rho_{k,l,m}(\nu)\big|^2w_{g_{k,l,m}}^2(\nu)\ \frac{d\nu}{\nu}\leq B\int_1^{e^{1/b}}\hspace{-3 mm}\big|\rho_{k,l,m}(\nu)\big|^2w_{g_{k,l,m}}(\nu)\ \frac{d\nu}{\nu}.
\end{align*}
Using \eqref{16} and \eqref{20} in the above equation, we get
\begin{align*}
A\big\|F\big\|_{L^2(\R^*,\b;d\kappa)}^2\leq\sum_{(k,l,m,p)\in\mathcal{F}}|\langle F,T_{e^{bp}}M_{a(2k,l,m)}\mathcal{G}\rangle|^2\leq B\big\|F\big\|_{L^2(\R^*,\b;d\kappa)}^2.
\end{align*}
Hence the required result follows from density argument.
\end{proof}
\begin{theorem}
Let $\{g_{k,l,m}(\cdot,\nu e^{s/b},j):k,l\in\Z,m\in\z\}$ be an orthogonal system in $L^2(\c^n)$ for $a.e.\ \nu\in(1,e^{1/b})$ and for each $j\in\{1,-1\},s\in\z$. Define 
$$w_{g_{k,l,m}}(\nu)=\sum_{j\in\{1,-1\}}\sum_{s\in\z}\|g_{k,l,m}(\cdot,\nu e^{s/b},j)\|_{L^2(\c^n)}^2,\  \nu\in(1,e^{1/b}).$$
Then the family $\mathcal{U}(\mathcal{G})$ is a Riesz sequence with bounds $A,B>0$ iff $Ab\leq w_{g_{k,l,m}}(\nu)\leq Bb$ for $a.e.\ \nu\in(1,e^{1/b})$, $\forall\ (k,l,m)\in\z^{2n+1}$.
\end{theorem}
\begin{proof}
Let $R:\ell^2(\z^{2n+2})\longrightarrow L^2(\R^*,\b;d\kappa)$ defined by
\begin{align*}
R\big(\{\alpha_{k,l,m,p}\}\big)=\sum_{(k,l,m,p)\in\mathcal{F}}\alpha_{k,l,m,p}T_{e^{bp}}M_{a(2k,l,m)}\mathcal{G},\hspace{4 mm}\forall\ \{\alpha_{k,l,m,p}\}\in\ell^2(\z^{2n+2}),
\end{align*}
be the synthesis operator corresponding to the system $\mathcal{U}(\mathcal{G})$. Then $\mathcal{U}(\mathcal{G})$ is a Riesz sequence with bounds $A,B>0$ iff
\begin{align*}
A\big\|\{\alpha_{k,l,m,p}\}\big\|_{\ell^2(\z^{2n+2})}^2\leq\big\|R\big(\{\alpha_{k,l,m,p}\}\big)\big\|_{L^2(\R^*,\b;d\kappa)}^2\leq B\big\|\{\alpha_{k,l,m,p}\}\big\|_{\ell^2(\z^{2n+2})}^2\numberthis\label{26}\ ,
\end{align*}
$\forall\ \{\alpha_{k,l,m,p}\}\in\ell^2(\z^{2n+2})$. By \eqref{16} we get
\begin{align*}
\big\|R\big(\{\alpha_{k,l,m,p}\}\big)\big\|_{L^2(\R^*,\b;d\kappa)}^2&=\sum_{(k,l,m)\in\z^{2n+1}}\int_1^{e^{1/b}}|\rho_{k,l,m}(\nu)|^2\ w_{g_{k,l,m}}(\nu)\ \frac{d\nu}{\nu}\numberthis\label{27}\ ,
\end{align*}
where $\rho_{k,l,m}(\nu)=\sum_{p\in\z}\alpha_{k,l,m,p}e^{-2\pi ibp\log\nu}$. Now
\begin{align*}
\widehat{\rho}_{k,l,m}(e^{-p})&=b\int_1^{e^{1/b}}\rho_{k,l,m}(\nu) e^{2\pi ibp\log\nu}\ \frac{d\nu}{\nu}\\
&=b\int_1^{e^{1/b}}\bigg(\sum_{p^\prime\in\z}\alpha_{k,l,m,p^\prime}e^{-2\pi ibp^\prime\log\nu}\bigg)e^{2\pi ibp\log\nu}\ \frac{d\nu}{\nu}\\
&=b\sum_{p^\prime\in\z}\alpha_{k,l,m,p^\prime}\int_1^{e^{1/b}}e^{2\pi ib(p-p^\prime)\log\nu}\ \frac{d\nu}{\nu}\\
&=b\sum_{p^\prime\in\z}\alpha_{k,l,m,p^\prime}\cdot\frac{1}{b}\ \delta_{p,p^\prime}\\
&=\alpha_{k,l,m,p}.
\end{align*}
Thus by using Plancherel formula for $L^2((1,e^{1/b});d\kappa)$, we have
\begin{align*}
\sum_{p\in\z}\big|\alpha_{k,l,m,p}\big|^2=b\int_1^{e^{1/b}}\big|\rho_{k,l,m}(\nu)\big|^2\ \frac{d\nu}{\nu}.\numberthis\label{28}
\end{align*} 
Using \eqref{27} and \eqref{28} in \eqref{26}, we get
\begin{align*}
Ab\int_1^{e^{1/b}}\big|\rho_{k,l,m}(\nu)\big|^2\ \frac{d\nu}{\nu}\leq\sum_{(k,l,m)\in\z^{2n+1}}\int_1^{e^{1/b}}|\rho_{k,l,m}(\nu)|^2\ w_{g_{k,l,m}}(\nu)\ \frac{d\nu}{\nu}\leq Bb\int_1^{e^{1/b}}\big|\rho_{k,l,m}(\nu)\big|^2\ \frac{d\nu}{\nu}.\numberthis\label{29}
\end{align*}
From the left hand side inequality of \eqref{29} we have
\begin{align*}
\sum_{(k,l,m)\in\z^{2n+1}}\int_1^{1/b}(Ab-w_{g_{k,l,m}}(\nu))|\rho_{k,l,m}(\nu)|^2\ \frac{d\nu}{\nu}\leq 0.\numberthis\label{30}
\end{align*}
Let $Q_{k,l,m}^{(1)}=\{\nu\in(1,e^{1/b}):w_{g_{k,l,m}}(\nu)<Ab\}$. Define $\rho^{(1)}=\{\rho_{k,l,m}^{(1)}w_{g_{k,l,m}}^{1/2}\}_{(k,l,m)\in\z^{2n+1}}$, where $\rho_{k,l,m}^{(1)}=\frac{1}{\beta_{k,l,m}\|\mathcal{G}_{k,l,m}\|_{L^2(\R^*,\b;d\kappa)}}\chi_{Q_{k,l,m}^{(1)}}$. Then a similar calculation as in Theorem \ref{a} leads to $\rho^{(1)}\in\ell^2\big(\z^{2n+1},\\L^2((1,e^{1/b});d\kappa)\big)$. Substituting $\rho=\rho^{(1)}$ in \eqref{30}, we get
\begin{align*}
\sum_{(k,l,m)\in\z^{2n+1}}\frac{1}{\beta_{k,l,m}^2\|\mathcal{G}_{k,l,m}\|_{L^2(\R^*,\b;d\kappa)}^2}\int_{Q_{k,l,m}^{(1)}}(Ab-w_{g_{k,l,m}}(\nu))\ \frac{d\nu}{\nu}\leq 0\ ,
\end{align*}
which in turn implies that
\begin{align*}
\int_{Q_{k,l,m}^{(1)}}(Ab-w_{g_{k,l,m}}(\nu))\ \frac{d\nu}{\nu}=0,\hspace{5 mm} \forall\ (k,l,m)\in\z^{2n+1}.
\end{align*}
Since $Ab-w_{g_{k,l,m}}(\nu)>0$ on $Q_{k,l,m}^{(1)}$, we get $|Q_{k,l,m}^{(1)}|=0$. From the right hand side inequality of \eqref{29} we have
\begin{align*}
\sum_{(k,l,m)\in\z^{2n+1}}\int_1^{1/b}(w_{g_{k,l,m}}(\nu)-Bb)|\rho_{k,l,m}(\nu)|^2\ \frac{d\nu}{\nu}\leq 0.\numberthis\label{31}
\end{align*}
Let $Q_{k,l,m}^{(2)}=\{\nu\in(1,e^{1/b}):w_{g_{k,l,m}}(\nu)>Bb\}$. As before substituting $\rho=\rho^{(2)}=\\ \{\rho_{k,l,m}^{(2)}w_{g_{k,l,m}}^{1/2}\}_{(k,l,m)\in\z^{2n+1}}$, where $\rho_{k,l,m}^{(2)}=\frac{1}{\beta_{k,l,m}\|\mathcal{G}_{k,l,m}\|_{L^2(\R^*,\b;d\kappa)}}\chi_{Q_{k,l,m}^{(2)}}$, in \eqref{31} we get $|Q_{k,l,m}^{(2)}|=0$. Therefore $Ab\leq w_{g_{k,l,m}}(\nu)\leq Bb$ for $a.e.\ \nu\in(1,e^{1/b})$ and for each $(k,l,m)\in\z^{2n+1}$.\\

Conversely assume that $Ab\leq w_{g_{k,l,m}}(\nu)\leq Bb$ for $a.e.\ \nu\in(1,e^{1/b})$ and for each $(k,l,m)\in\z^{2n+1}$. Then \eqref{29} follows immediately.   
\end{proof}
\section*{Acknowledgement}                                                                                                                                                                                                                                                                                                                                                                                                                                                                                                                                                                                                                                                                                                                                                                                                                                                                                                                                                                                                                                                                                                                                                                                                                                                                                                                                                                                                                                                                                                                                                                                                                                                                                                                                                                                                                                                                                                                                                                                                                                                                                                                                                                                                                                                                                                                                                                                                                                                                                                                                                                                                                                                                                                                                                                                                                                                                                                                                                                                                                                                                                                                                                                                                                                                                                                                                                                                                                                                                                                                                                                                                                                                                                                                                                                                                                                                                                                                                                                                                                                                                                                                                                                                                                                     
One of the authors (R.R) was partially supported by Visiting professor program funded by the Bavarian State Ministry for Science, Research and the Arts, TUM international centre for the stay at Technical University Munich, Germany for the period October to December 2019. She sincerely thanks Professor Massimo Fornasier and Professor Peter Massopust , TUM, for their kind invitation and excellent hospitality for the entire period of visit during October 2019 to September  2020.\\      

The authors profusely thank Prof. K. Parthasarathy, RIASM for his helpful suggestions regarding the explicit calculations for the Fourier transform on $\R^*$.

\bibliographystyle{amsplain}
\bibliography{gref}

\providecommand{\bysame}{\leavevmode\hbox to3em{\hrulefill}\thinspace}
\providecommand{\MR}{\relax\ifhmode\unskip\space\fi MR }
\providecommand{\MRhref}[2]{%
  \href{http://www.ams.org/mathscinet-getitem?mr=#1}{#2}
}
\providecommand{\href}[2]{#2}
\begin{thebibliography}{10}

\bibitem{arati}
S.~Arati and R.~Radha, \emph{Frames and {R}iesz bases for shift invariant
  spaces on the abstract {H}eisenberg group}, Indag. Math. (N.S.) \textbf{30}
  (2019), no.~1, 106--127. \MR{3906124}

\bibitem{aratijmpa}
S.~Arati and R.~Radha, \emph{Orthonormality of wavelet system on the
  {H}eisenberg group}, J. Math. Pures Appl. (9) \textbf{131} (2019), 171--192.
  \MR{4021173}

\bibitem{araticol}
S.~Arati and R.~Radha, \emph{Wavelet system and {M}uckenhoupt {$A_2$} condition
  on the {H}eisenberg group}, Colloq. Math. \textbf{158} (2019), no.~1, 59--76.
  \MR{3999453}

\bibitem{cabrelli}
C.~Cabrelli and V.~Paternostro, \emph{Shift-invariant spaces on {LCA} groups},
  J. Funct. Anal. \textbf{258} (2010), no.~6, 2034--2059. \MR{2578463}

\bibitem{olebook}
O.~Christensen, \emph{An introduction to frames and {R}iesz bases}, second ed.,
  Applied and Numerical Harmonic Analysis, Birkh\"{a}user/Springer, [Cham],
  2016. \MR{3495345}

\bibitem{ole}
O.~Christensen and S.S. Goh, \emph{Fourier-like frames on locally compact
  abelian groups}, J. Approx. Theory \textbf{192} (2015), 82--101. \MR{3313475}

\bibitem{CMO}
B.~Currey, A.~Mayeli, and V.~Oussa, \emph{Characterization of shift-invariant
  spaces on a class of nilpotent {L}ie groups with applications}, J. Fourier
  Anal. Appl. \textbf{20} (2014), no.~2, 384--400. \MR{3200927}

\bibitem{dahlke}
S.~Dahlke, \emph{Multiresolution analysis and wavelets on locally compact
  abelian groups}, Wavelets, images, and surface fitting
  ({C}hamonix-{M}ont-{B}lanc, 1993), A K Peters, Wellesley, MA, 1994,
  pp.~141--156. \MR{1302244}

\bibitem{farkov}
Yu.~A. Farkov, \emph{Orthogonal wavelets with compact supports on locally
  compact abelian groups}, Izv. Ross. Akad. Nauk Ser. Mat. \textbf{69} (2005),
  no.~3, 193--220. \MR{2150505}

\bibitem{follandphase}
G.B. Folland, \emph{Harmonic analysis in phase space}, Annals of Mathematics
  Studies, vol. 122, Princeton University Press, Princeton, NJ, 1989.
  \MR{983366}

\bibitem{follandabs}
G.B. Folland, \emph{A course in abstract harmonic analysis}, Studies in
  Advanced Mathematics, CRC Press, Boca Raton, FL, 1995. \MR{1397028}

\bibitem{holsch}
M.~Holschneider, \emph{Wavelet analysis over abelian groups}, Appl. Comput.
  Harmon. Anal. \textbf{2} (1995), no.~1, 52--60. \MR{1313098}

\bibitem{iverson}
J.W. Iverson, \emph{Frames generated by compact group actions}, Trans. Amer.
  Math. Soc. \textbf{370} (2018), no.~1, 509--551. \MR{3717988}

\bibitem{jakobsen}
M.S. Jakobsen and J.~Lemvig, \emph{Co-compact {G}abor systems on locally
  compact abelian groups}, J. Fourier Anal. Appl. \textbf{22} (2016), no.~1,
  36--70. \MR{3448915}

\bibitem{kamyabi}
R.~A. Kamyabi~Gol and R.R. Tousi, \emph{The structure of shift invariant spaces
  on a locally compact abelian group}, J. Math. Anal. Appl. \textbf{340}
  (2008), no.~1, 219--225. \MR{2376149}

\bibitem{gitta}
G.~Kutyniok and D.~Labate, \emph{The theory of reproducing systems on locally
  compact abelian groups}, Colloq. Math. \textbf{106} (2006), no.~2, 197--220.
  \MR{2283810}

\bibitem{mayelimra}
A.~Mayeli, \emph{Shannon multiresolution analysis on the {H}eisenberg group},
  J. Math. Anal. Appl. \textbf{348} (2008), no.~2, 671--684. \MR{2445768}

\bibitem{saswata}
R.~Radha and S.~Adhikari, \emph{Frames and {R}iesz bases of twisted
  shift-invariant spaces in {$L^2(\Bbb{R}^{2n})$}}, J. Math. Anal. Appl.
  \textbf{434} (2016), no.~2, 1442--1461. \MR{3415732}

\bibitem{saswatahouston}
R.~Radha and S.~Adhikari, \emph{Shift-invariant spaces with countably many
  mutually orthogonal generators on the {Heisenberg} group}, Houston J. Math.
  (2020), no.~46, 435--463.

\bibitem{saswatacollec}
R.~Radha and Saswata Adhikari, \emph{Left translates of a square integrable
  function on the {H}eisenberg group}, Collect. Math. \textbf{71} (2020),
  no.~2, 239--262. \MR{4083644}

\bibitem{radhas}
R.~Radha and N.S. Kumar, \emph{Shift invariant spaces on compact groups}, Bull.
  Sci. Math. \textbf{137} (2013), no.~4, 485--497. \MR{3054272}

\bibitem{rudin}
W.~Rudin, \emph{Fourier analysis on groups}, Interscience Tracts in Pure and
  Applied Mathematics, No. 12, Interscience Publishers (a division of John
  Wiley and Sons), New York-London, 1962. \MR{0152834}

\bibitem{thangavelu}
S.~Thangavelu, \emph{Harmonic analysis on the {H}eisenberg group}, Progress in
  Mathematics, vol. 159, Birkh\"{a}user Boston, Inc., Boston, MA, 1998.
  \MR{1633042}

\end{thebibliography}
\end{document}